\documentclass[12pt,reqno]{amsart}

\usepackage[pdfauthor   = {{Martin\ Bies} {Sebastian\ Posur}},
            pdftitle    = {Tensor\ products\ of\ finitely\ presented\ functors},
            pdfsubject  = {},
            pdfkeywords = {Freyd category;\ finitely presented functor;\ computable abelian category;\ monoidal structure},
            bookmarks=true,
            bookmarksopen=true,
            pagebackref=true,
            hyperindex=true,
            colorlinks=true,
            linkcolor=blue,
            citecolor=blue,
            filecolor=blue,
            urlcolor=blue,
            ]{hyperref}

\usepackage[utf8]{inputenc} 
\usepackage[T1]{fontenc}
\usepackage[english]{babel}

\usepackage{a4wide,
            amsmath,
            amssymb,
            amsthm,
            array,
            caption,
            colortbl,
            enumerate,
            etex,
            extarrows,
            fancyvrb,
            float,
            graphicx,
            latexsym,
            listings,
            lmodern,
            mathabx,
            mathdots, 
            mathrsfs,
            mathtools,
            mdwlist,
            multirow,
            pdflscape,
            rotating,
            stmaryrd,
            verbatim,
            xcolor,
            xspace,
            wrapfig
            }

\usepackage[sort&compress,capitalise]{cleveref}

\usepackage{accents} 
\usepackage[toc,page]{appendix}
\usepackage{enumitem}
\usepackage{tikz}
\usepackage{tikz-cd}
\usepackage[colorinlistoftodos,shadow]{todonotes}
\usepackage[all]{xy}

\usetikzlibrary{automata,shapes,arrows,matrix,backgrounds,positioning,plotmarks,calc,patterns,matrix,decorations.pathreplacing,decorations.pathmorphing,decorations.text,decorations.markings}
\tikzset{round left paren/.style={ncbar=0.5cm,out=120,in=-120}}
\tikzset{round right paren/.style={ncbar=0.5cm,out=60,in=-60}}

\definecolor{lightgray}{gray}{0.8}

\definecolor{ctcolor}{gray}{0.95}

\definecolor{ctucolor}{gray}{0.85}

\newlist{theoremenumerate}{enumerate}{1}
\setlist[theoremenumerate]{label=(\arabic{theoremenumeratei}), ref=\thetheorem.(\arabic{theoremenumeratei}),noitemsep}

\newtheoremstyle{mytheoremstyle}
    {5pt}         
    {5pt}         
    {\itshape}    
    {\parindent}  
    {\bf}         
    {.}           
    {.5em}        
    {}            

\theoremstyle{mytheoremstyle}

\newtheorem{theorem}{Theorem}[subsection]

\newtheorem{lemma}[theorem]{Lemma}

\theoremstyle{remark}
\newtheorem{remark}[theorem]{Remark}

\newtheoremstyle{mytdefintionstyle}
    {5pt}         
    {5pt}         
    {\rm}         
    {\parindent}  
    {\bf}         
    {.}           
    {.5em}        
    {}            

\theoremstyle{mytdefintionstyle}

\newtheorem{definition}[theorem]{Definition}

\newtheorem{example}[theorem]{Example}

\newtheorem{construction}[theorem]{Construction}

\newtheorem{notation}[theorem]{Notation}
\newtheorem*{notationnonumber}{Notation}

\newtheorem*{convention}{Convention}

\newcolumntype{C}[1]{>{\centering\arraybackslash$}p{#1}<{$}}
\newlength{\mycolwd}
\newcolumntype{L}{>{\raggedleft}p{0.28\textwidth}}
\newcolumntype{R}{p{0.8\textwidth}}

\makeatletter
\newcommand{\thickhline}{%
    \noalign {\ifnum 0=`}\fi \hrule height 1pt
    \futurelet \reserved@a \@xhline
}
\newcolumntype{"}{@{\hskip\tabcolsep\vrule width 1pt\hskip\tabcolsep}}
\makeatother

\definecolor{ExQ}{HTML}{0000FF}
\definecolor{Dec}{HTML}{E07B00}

\newcommand{\AC}{\mathbf{A}}
\newcommand{\BC}{\mathbf{B}}
\newcommand{\XC}{\mathbf{X}}

\newcommand{\PC}{\mathbf{P}}

\newcommand{\pmatrow}[2]{ \begin{pmatrix}{#1} & {#2} \end{pmatrix} }
\newcommand{\pmatcol}[2]{ \begin{pmatrix}{#1} \\ {#2} \end{pmatrix} }

\newcommand{\N}{\mathbb{N}}
\newcommand{\Z}{\mathbb{Z}}

\newcommand{\nindex}{\underline{n}}

\newcommand{\Ab}{\mathbf{Ab}}

\newcommand{\op}{\mathrm{op}}

\newcommand{\FreydName}{\mathrm{Freyd}}
\newcommand{\Freyd}{\mathcal{A}}
\DeclareMathOperator{\oF}{\otimes_{\That}}
\newcommand{\unitF}{1_{\FreydName}}

\newcommand{\id}{\mathrm{id}}

\DeclareMathOperator{\Hom}{\mathrm{Hom}}
\DeclareMathOperator{\Homlin}{\mathrm{Hom} }  
\DeclareMathOperator{\Homlincok}{\Hom^{\mathrm{r}}}

\DeclareMathOperator{\Mor}{\mathrm{Mor}}
\DeclareMathOperator{\IHom}{\mathrm{\underline{Hom}}}

\DeclareMathOperator{\kernel}{\mathrm{ker}}

\DeclareMathOperator{\cokernel}{\mathrm{coker}}

\newcommand{\emb}{\mathrm{emb}}

\newcommand{\fpmodl}{\text{-}\mathrm{fpmod}}

\newcommand{\fpf}{\mathrm{fp}}
\newcommand{\Rows}{\mathrm{Rows}}

\newcommand*{\ie}{i.e.\@\xspace}

\newcommand*{\eg}{e.g.\@\xspace}

\newcommand{\fp}{f.\,p.\@\xspace}
\newcommand{\Fp}{F.\,p.\@\xspace}

\DeclareMathOperator{\T}{\mathrm{T}}
\DeclareMathOperator{\That}{\mathrm{ \text{$\widehat{T}$}}}
\DeclareMathOperator{\Ass}{\mathrm{\Pi}}
\DeclareMathOperator{\Asshat}{\mathrm{ \text{$\widehat{\Pi}$}}}
\DeclareMathOperator{\Br}{\mathrm{\Gamma}}
\DeclareMathOperator{\Brhat}{\mathrm{ \text{$\widehat{\Gamma}$}}}
\DeclareMathOperator{\LU}{\mathrm{\Psi}}
\DeclareMathOperator{\LUhat}{\mathrm{ \text{$\widehat{\Psi}$}}}
\DeclareMathOperator{\RU}{\mathrm{\Omega}}
\DeclareMathOperator{\RUhat}{\mathrm{ \text{$\widehat{\Omega}$}}}

\DeclareMathOperator{\PreEv}{\mathrm{ev}}
\newcommand{\PreEvhat}{\widehat{\mathrm{ev}}}
\DeclareMathOperator{\PreCoev}{\mathrm{coev}}
\newcommand{\PreCoevhat}{\widehat{\mathrm{coev}}}

\author{Martin Bies}
\thanks{The work of M. Bies is supported by the \emph{Wiener-Anspach foundation}. M. Bies thanks the \emph{University of Siegen} and the \emph{GAP Singular Meeting and School} for hospitality during this project.}
\address{Service de Physique Théorique et Mathématique, Université Libre de Bruxelles and International Solvay Institutes, Campus Plaine C.P. 231, B-1050 Bruxelles, Belgium}
\email{\href{mailto:Martin Bies <martin.bies@ulb.ac.be>}{martin.bies@ulb.ac.be}}

\author{Sebastian Posur}
\thanks{The work of S. Posur is supported by Deutsche Forschungsgemeinschaft (DFG) grant SFB-TRR 195: \emph{Symbolic Tools in Mathematics and their Application}.}
\address{Department of mathematics, University of Siegen, 57068 Siegen, Germany}
\email{\href{mailto:Sebastian Posur <sebastian.posur@uni-siegen.de>}{sebastian.posur@uni-siegen.de}}

\begin{document}

\title[Tensor products of finitely presented functors]{Tensor products of finitely presented functors}

\begin{abstract} 
We study right exact tensor products on the category of finitely presented functors.
As our main technical tool, we use a multilinear version of the universal property of so-called Freyd categories.
Furthermore, we compare our constructions with the Day convolution of arbitrary functors.
Our results are stated in a constructive way and give a unified approach for the implementation
of tensor products in various contexts.
\end{abstract}

\maketitle

\tableofcontents

\section{Introduction}
Let $\AC$ be a small additive category. In this paper, we study right exact tensor products on the category $\fpf( \AC^{\op}, \Ab )$
of finitely presented functors $\AC^{\op} \rightarrow \Ab$ with natural transformations as morphisms,
where $\Ab$ denotes the category of abelian groups.
As our main computational tool, we use the so-called Freyd category \cite{FreydRep}, \cite{BelFredCats} $\Freyd( \AC )$,
that comes with a functor $\AC \rightarrow \Freyd( \AC )$ which is universal among all functors from $\AC$ mapping into a category
with cokernels. There is an equivalence $\Freyd( \AC ) \simeq \fpf( \AC^{\op}, \Ab )$ which allows us to state 
all results on finitely presented functors in the language of Freyd categories.
We prefer to use the language of Freyd categories in this paper due to its inherent constructiveness \cite{PosFreyd}.

In \Cref{sec:2CatUniPropOfFreyd}, we give a short introduction to the theory of Freyd categories.
After that, we prove a multilinear $2$-categorical version of their universal property in \Cref{theorem:2_cat_universal_freyd_mult},
which will turn out to be our main technical tool in this paper.

In \Cref{sec:FPMonoidalStructuresOnFreyd}, we introduce the notion of finitely presented promonoidal structures on $\AC$,
and show how they can be extended to right exact monoidal structures on $\Freyd( \AC )$.
In short, finitely presented promonoidal structures can be seen as restrictions of right exact monoidal structures on $\Freyd( \AC )$
to the subcategory $\AC \subseteq \Freyd( \AC )$.
At this point, the proofs for checking the monoidal identities like the pentagon identity can be carried out quite
conveniently due to our multilinear version of the universal property of Freyd categories.

Originally, promonoidal structures on additive categories were introduced by Brian Day in \cite{Day70} and \cite{Day74}.
He showed that promonoidal structures give rise to closed monoidal structures on the category $\Hom( \AC, \Ab )$ of all additive functors $\AC \rightarrow \Ab$.
Since $\Freyd( \AC^{\op } ) \simeq \fpf( \AC, \Ab ) \subseteq \Hom( \AC, \Ab )$,
in \Cref{subsec:ConnectionToDay} we discuss how our construction can be seen as a restriction
of the so-called Day convolution of arbitrary functors to finitely presented ones (\Cref{theorem:restriction_day}).

As applications of the theory presented in this paper,
we show how our findings can be used to induce right exact tensor products on iterated Freyd categories,
and in particular on free abelian categories,
a topic that is also treated in \cite{TensorNori}.

Since the derivation of monoidal structures on $\Freyd( \AC )$ from finitely presented promonoidal structures on $\AC$ is completely constructive,
we complete this paper by presenting the corresponding constructions explicitly.
In particular, this provides a computationally unified approach to tensor products of \fp modules and \fp graded modules,
that can both be interpreted as special instances of categories of \fp functors \cite{PosFreyd}.
A computer implementation of the special case where the promonoidal structure on $\AC$
is actually monoidal is realized within the \texttt{CAP}-project,
a software project for constructive category theory \cite{CAP-project}, \cite{GP_Rundbrief}, \cite{GPSSyntax},
as an own package for \emph{Freyd categories} \cite{FreydCategoriesPackage}.

\begin{convention}
      In this paper, $\Ab$ denotes the category of abelian groups.
      For the precomposition of two morphisms $\alpha: a \rightarrow b$, $\beta: b \rightarrow c$,
      we write $\alpha \cdot \beta$, for their postcomposition, we write $\beta \circ \alpha$.
      Morphisms between direct sums are written in matrix notation, where we use the row convention,
      \eg, a morphism $a \rightarrow b \oplus c$ is written as a $1 \times 2$ row $\pmatrow{\gamma}{\delta}$
      with entries given by morphisms $\gamma: a \rightarrow b$, $\delta: a \rightarrow c$.
\end{convention}

\section{Freyd categories and their universal property} \label{sec:2CatUniPropOfFreyd}

The Freyd category $\Freyd( \AC )$ of an additive category $\AC$ is a universal way of equipping $\AC$ with cokernels.
In this section, we give an introduction to the theory of Freyd categories and prove how we can lift multilinear functors
and natural transformations from $\AC$ to $\Freyd( \AC )$.

\subsection{Preliminaries: Freyd categories}

Let $\AC$ be an additive category.
We recall well-known results about the Freyd category of $\AC$,
that can be found in the original source \cite{FreydRep} by Freyd,
or in \cite{BelFredCats} by Beligiannis (who coined the term \emph{Freyd category}).
A constructive approach to Freyd categories is given in \cite{PosFreyd}.

\begin{definition}
      The \textbf{Freyd category } $\Freyd( \AC )$ consists of the following data:
      \begin{enumerate}
            \item Objects are given by morphisms in $\AC$.
            \item Morphisms from $(a \stackrel{\rho}{\longleftarrow} r)$ to $(a' \stackrel{\rho'}{\longleftarrow} r')$
        are given by morphisms
        $a \stackrel{\alpha}{\longrightarrow} a'$ in $\AC$
        such that there exists another morphism $r \stackrel{\omega}{\longrightarrow} r'$ 
        rendering the diagram
        \begin{center}
         \begin{tikzpicture}[label/.style={postaction={
          decorate,
          decoration={markings, mark=at position .5 with \node #1;}},
          mylabel/.style={thick, draw=none, align=center, minimum width=0.5cm, minimum height=0.5cm,fill=white}}]
          \coordinate (r) at (2.5,0);
          \coordinate (u) at (0,2);
          \node (A) {$a$};
          \node (B) at ($(A)+(r)$) {$r$};
          \node (C) at ($(A) - (u)$) {$a'$};
          \node (D) at ($(B) - (u)$) {$r'$};
          \draw[->,thick] (B) --node[above]{$\rho$} (A);
          \draw[->,thick] (D) --node[above]{$\rho'$} (C);
          \draw[->,thick] (A) --node[left]{$\alpha$} (C);
          \draw[->,thick,dashed] (B) --node[right]{$\omega$} (D);
          \end{tikzpicture}
        \end{center}
        commutative. 
        The morphism $\alpha$ is called the \textbf{morphism datum} and any possible $\omega$ a \textbf{morphism witness}.
        We write $\{\alpha\}$ or $\{\alpha, \omega\}$ for a morphism in $\mathcal{A}(\PC)$,
        depending on whether we would like to highlight a particular choice of a morphism witness.
        Given another morphism $\{\alpha'\}$ from $(a \stackrel{\rho}{\longleftarrow} r)$ to $(a' \stackrel{\rho'}{\longleftarrow} r')$,
        we define it
        to be \textbf{equal in $\Freyd(\PC)$} to $\{\alpha\}$
        if there exists a $\lambda: a \rightarrow r'$ 
        such that
        \[
              ( \alpha - \alpha' ) = \lambda \cdot \rho'.
        \]
        Any such $\lambda$ is called  a \textbf{witness for $\alpha$ and $\alpha'$ being equal}.
  \item Composition, identities, and the additive structure are inherited from $\AC$.
      \end{enumerate}
\end{definition}

\begin{notation}
      We will usually refer to objects 
      in $\AC$ by small letters $a,b,c, \dots$
      and to objects
      in $\Freyd( \AC )$ by capital letters $A,B,C, \dots$.
      Moreover, whenever we want to access the underlying morphism datum
      of a variable representing an object in $\Freyd( \AC )$,
      we use the notation
      $A = (a \xleftarrow{\rho_a} r_a)$,
      $B = (b \xleftarrow{\rho_b} r_b)$,
      $C = (c \xleftarrow{\rho_c} r_c)$, \dots .
\end{notation}

\begin{remark}[Implicit embeddings]\label{remark:implicit_emb}
      We have a functor
      \begin{align*}
            \emb: \AC \longrightarrow \Freyd( \AC ):
            a \mapsto (a \leftarrow 0) \, .
      \end{align*}
      Every object $a \in \AC$ can be canonically understood as $\mathrm{emb} ( a ) = (a \leftarrow 0 ) \in \Freyd( \AC )$. Whenever the distinction between these two objects is important, we will indeed use this distinguished notation. However, if there is no danger of confusion, we will not use this distinguished notation, in favour of a simplified notation. In the latter case it should always be clear from the context if we are using $a \in \AC$ or $\mathrm{emb} ( a ) \in \Freyd( \AC )$.
\end{remark}

\begin{construction}\label{construction:cokernels_in_freyd}
       The category $\Freyd( \AC )$ has cokernels.
       The cokernel object of a given morphism
       $\{\alpha\}: (a \xleftarrow{\rho_a} r_a) \longrightarrow (b \xleftarrow{\rho_b} r_b)$
       can be constructed as
       \[
            \cokernel( \{\alpha\} ) := (b \xleftarrow{\pmatcol{\rho_b}{\alpha}} r_b \oplus a) \, .
       \]
\end{construction}

It is not necessarily the case that $\Freyd( \AC )$ has kernels.
Recall that a weak kernel of a morphism $\alpha: a \rightarrow b$ in $\AC$
is a morphism $\iota: k \rightarrow a$ in $\AC$ such that $\iota \cdot \alpha = 0$,
and for any other morphism $\tau: t \rightarrow a$ with $\tau \cdot \alpha = 0$,
there exists a \emph{not necessarily unique} morphism $\lambda: t \rightarrow k$
such that $\lambda \cdot \iota = \tau$.
\begin{theorem}[Freyd \cite{FreydRep}]\label{theorem:freyd}
      $\Freyd( \AC )$ 
      is abelian if and only if 
      $\AC$ admits weak kernels.
\end{theorem}

Let $\AC$ be a small\footnote{A category is small if its object class is a set.} additive category.
A covariant functor $\AC \xrightarrow{F} \Ab$ is called \textbf{finitely presented}
if it arises as the cokernel of a natural transformation between representable functors,
\ie, if there exists a morphism $\alpha: a \rightarrow b$ in $\AC$ such that
\[
      F \simeq \cokernel \left[ \Hom( \alpha, - ): \Hom(b,-) \rightarrow \Hom(a,-) \right] \, .
\]
Finitely presented functors form a full subcategory $\fpf( \AC, \Ab )$ of the category of all additive functors $\Hom( \AC, \Ab )$
from $\AC$ to $\Ab$
with natural transformations as morphisms.
Dually, we obtain a category $\fpf( \AC^{\op}, \Ab )$ consisting of all contravariant finitely presented functors.

\begin{theorem}\label{theorem:freyd_as_fp_functors}
      The operation
      \begin{align*}
            \Freyd( \AC ) &\longrightarrow \fpf( \AC^{\op}, \Ab ) \\
            (a \xleftarrow{\rho_a} r_a) &\longmapsto \cokernel( \Hom_{\AC}( -, \rho_a ) )
      \end{align*}
      gives rise to an equivalence of categories.
\end{theorem}

\subsection{Multilinear functors}

Let $n \in \N$ and $\AC_1, \dots, \AC_n, \BC$ be additive categories.
\begin{definition}
      We call an $n$-ary functor
      \[
            F: \AC_1 \times \dots \times \AC_n \rightarrow \BC
      \]
      which is componentwise additive a \textbf{multilinear functor}.
      Moreover, if $\AC_1, \dots, \AC_n, \BC$ have cokernels,
      we say that $F$ is \textbf{right exact}
      if it is componentwise right exact, \ie,
      if for all $k \in \{1, \dots, n\}$
      the natural morphism
      \[
            \cokernel\left( F( a_1, \dots, a_{k-1}, \alpha_k, a_{k+1}, \dots, a_n ) \right)
            \longrightarrow
            F\left( a_1, \dots, a_{k-1}, \cokernel( \alpha_k ), a_{k+1}, \dots, a_n \right)
      \]
      is an isomorphism
      for all objects
      $a_i \in \AC_i$, $i \neq k$, and all morphisms $\alpha_k \in \AC_k$.
\end{definition}

\begin{notationnonumber}
      For any index set $J$, whenever we have a $J$-indexed family of elements $(a_j)_{j \in J}$,
      we use the abbreviation $a_J$.
      We set $\nindex := \{1, \dots n \}$ since we will often use this special index set.
      If we wish to evaluate $F$ at $a_1, \dots, a_n$ but with the $j$-th component substituted with the element $b$, we write
      \[
            F( a_{{\nindex}-j}; b ) := F( (a_i)_{i \in {\nindex}-j}; b ) := F( a_1, \dots, a_{j-1}, b, a_{j+1}, \dots, a_{n}) \, .
      \]
\end{notationnonumber}

Right exactness of a multilinear functor can be conveniently rephrased as follows.

\begin{lemma}\label{lemma:criterion_resp_cok}
      A multilinear functor $F: \AC_1 \times \dots \times \AC_n \rightarrow \BC$
      is right exact if and only if
      for any tuple of morphisms $( b_i \stackrel{\alpha_i}{\longrightarrow} a_i )_{i \in {\nindex}} \in \prod_{i \in {\nindex}} \AC_i$
      the sequence
      \begin{center}
            \begin{tikzpicture}[
                  baseline=(current bounding box.center),
                  label/.style={postaction={
            decorate,
            decoration={markings, mark=at position .5 with \node #1;}},
            mylabel/.style={thick, draw=none, align=center, minimum width=0.5cm, minimum height=0.5cm,fill=white}}]
            \coordinate (r) at (2.5,0);
            \coordinate (u) at (0,2);
            \node (D) {$0$};
            \node (C) at ($(D) + (r)$) {$F\big(  \cokernel \alpha_{{\nindex}} \big)$};
            \node (B) at ($(C)+(r)$) {$F\big( a_{{\nindex}} \big)$};
            \node (A) at ($(B)+2.5*(r)$) {$\bigoplus_{j \in {\nindex}} F\big( a_{{\nindex}-j};b_j \big)$};
            
            \draw[->,thick] (A) --node[above]{$\big( F( \id_{a_{{\nindex}-j}}; \alpha_j )  \big)_{j \in {\nindex}}$} (B);
            \draw[->,thick] (B) -- (C);
            \draw[->,thick] (C) -- (D);
            \end{tikzpicture}
        \end{center}
      is exact.
\end{lemma}
\begin{proof}
      If we know that the sequence is exact for all tuples of morphisms,
      then evaluating at the special tuples
      \[
      ( 0 \rightarrow a_1, \dots, 0 \rightarrow a_{k-1}, b_{k} \xrightarrow{\alpha_k} a_{k}, 0 \rightarrow a_{k+1}, \dots 0 \rightarrow a_{n})
      \]
      proves that $F$ is right exact.
      
      For the converse, we proceed by induction on $n$. The case $n = 1 $ is trivial.
      For the induction step, let $F$ be an $(n+1)$-ary multilinear right exact functor.
      We set $A_i := \cokernel( \alpha_i )$
      for $i = 1, \dots, n+1$.
      Moreover, we write $F(-|x)$
      for the $n$-ary right exact functor obtained
      by fixing the last component of $F$ to a given object $x \in \AC_{n+1}$.
      Now, the claim follows by a diagram chase using
      the following diagram
      whose top row is exact by induction hypothesis
      and whose columns are exact since $F$ is right exact:
      \begin{center}
            \begin{tikzpicture}[label/.style={postaction={
                  decorate,
                  decoration={markings, mark=at position .5 with \node #1;}},
                  mylabel/.style={thick, draw=none, align=center, minimum width=0.5cm, minimum height=0.5cm,fill=white}}]
                  \coordinate (r) at (4,0);
                  \coordinate (rs) at (2.5,0);
                  \coordinate (d) at (0,-2);
                  \coordinate (ds) at (0,-1.5);
                  \node (A1) {};
                  \node (A2) at ($(A1) + (rs)$) {};
                  \node (A3) at ($(A2) + (r)$) {$0$};
                  \node (A4) at ($(A3) + (r)$) {$0$};
                  
                  \node (B1) at ($(A1) + (ds)$) {$0$};
                  \node (B2) at ($(B1) + (rs)$) {$F_{}( A_{{\nindex}} | A_{n+1})$};
                  \node (B3) at ($(B2) + (r)$) {$F( a_{{\nindex}} | A_{n+1} )$};
                  \node (B4) at ($(B3) + (r)$) {$\oplus_{j \in \nindex} F(a_{\nindex -j}; r_{a_j} | A_{n+1})$};
                  
                  \node (C1) at ($(B1) + (d)$) {};
                  \node (C2) at ($(C1) + (rs)$) {};
                  \node (C3) at ($(C2) + (r)$) {$F(a_{\nindex}|a_{n+1})$};
                  \node (C4) at ($(C3) + (r)$) {$\oplus_{j \in \nindex} F(a_{\nindex -j};r_{a_j}|a_{n+1})$};
                  
                  \node (D1) at ($(C1) + (d)$) {};
                  \node (D2) at ($(D1) + (rs)$) {};
                  \node (D3) at ($(D2) + (r)$) {$F(a_{\nindex}|r_{a_{n+1}})$};
                  \node (D4) at ($(D3) + (r)$) {$\oplus_{j \in \nindex} F(a_{\nindex -j};r_{a_j}|r_{a_{n+1}})$};
                  
                  \draw[->,thick] (B2) -- (B1);
                  \draw[->,thick] (B3) -- (B2);
                  \draw[->,thick] (B4) -- (B3);
                  
                  \draw[->,thick] (C4) -- (C3);
                  
                  \draw[->,thick] (D4) -- (D3);
                  
                  \draw[->,thick] (D4) -- (C4);
                  \draw[->,thick] (C4) -- (B4);
                  \draw[->,thick] (B4) -- (A4);

                  \draw[->,thick] (D3) -- (C3);
                  \draw[->,thick] (C3) -- (B3);
                  \draw[->,thick] (B3) -- (A3);
                  \end{tikzpicture} \qedhere
      \end{center}
\end{proof}

\subsection{The multilinear \texorpdfstring{$2$}{2}-categorical universal property of Freyd categories}\label{subsection:universal_prop_freyd}

Let $n \in \N$ and $\AC_1, \dots, \AC_n$ be additive categories.
Moreover, let $\BC$ be an additive category with cokernels.
We denote by $\Homlin( (\AC_i)_{i \in {\nindex}}, \BC)$
the category whose objects are multilinear functors of the form $\prod_{i \in {\nindex}} \AC_i \xrightarrow{F} \BC$
and whose morphisms are given by natural transformations.
Moreover, we denote by $\Homlincok( (\AC_i)_{i \in {\nindex}}, \BC)$
the category whose objects are multilinear functors $F$ as above
that are right exact
and whose morphisms are given by natural transformations.

\begin{theorem}[Multilinear $2$-categorical universal property of Freyd categories]\label{theorem:2_cat_universal_freyd_mult}
There is an equivalence of categories
\[ \Homlin ( (\AC_i)_{i \in {\nindex}}, \BC ) \simeq \Homlincok\left( (\Freyd( \AC_i ) )_{i \in {\nindex}}, \BC \right) \, . \]
\end{theorem}

The goal of this Subsection is to prove this universal property
by providing explicit constructions that define the stated equivalence of categories.

\begin{construction}\label{construction:extension_of_multi_functor}
      Suppose given a multilinear functor
      $F: \prod_{i \in {\nindex}} \AC_i \rightarrow \BC$,
      we can construct a functor
      \[
            \widehat{F}: \prod_{i \in {\nindex}} \Freyd( \AC_i ) \longrightarrow \BC
      \]
      by setting
      \[ 
            \widehat{F} \big( A_{{\nindex}} \big) := 
            \mathrm{cok} \left( F(a_{{\nindex}}) \xleftarrow{\left( F( \id_{a_{{\nindex}-j}};\rho_{a_j})\right)_{j \in {\nindex}}} \bigoplus_{j \in {\nindex}} F( a_{{\nindex}-j}; r_{a_j} ) \right)
      \]
      for objects $A_{{\nindex}} \in \prod_{i \in {\nindex}} \Freyd( \AC_i )$.
      
      For morphisms $( A_i \xrightarrow{\{\alpha_i, \omega_i\}} B_i)_{i \in {\nindex}}$ we define the action of $\widehat{F}$
      via the following commutative diagram with exact rows:
\begin{center}
\begin{tikzpicture}[baseline=(current bounding box.center),label/.style={postaction={
          decorate,
          decoration={markings, mark=at position .5 with \node #1;}},
          mylabel/.style={thick, draw=none, align=center, minimum width=0.5cm, minimum height=0.5cm,fill=white}}]

          \def\w{3.5};
          \def\h{3};

          \node (A) at (-2*\w,0) {$0$};
          \node (B) at (-\w,0) {$\widehat{F} \big( A_{{\nindex}} \big)$};
          \node (C) at (0,0) {$F(a_{{\nindex}})$};
          \node (D) at (2*\w,0) {$\bigoplus_{j \in {\nindex}} F( a_{{\nindex}-j}; r_{a_j} )$};
          \node (E) at (-2*\w,-\h) {$0$};
          \node (F) at (-\w,-\h) {$\widehat{F} \big( B_{{\nindex}} \big)$};
          \node (G) at (0,-\h) {$F(b_{{\nindex}})$};
          \node (H) at (2*\w,-\h) {$\bigoplus_{j \in {\nindex}} F( b_{{\nindex}-j}; r_{b_j} )$};

          \draw[->,thick] (B) --node[below] {} (A);
          \draw[->,thick] (C) --node[below] {} (B);
          \draw[->,thick] (D) --node[above] {$\left(F( \id_{a_{{\nindex}-j}};\rho_{a_j})\right)_{j \in {\nindex}}$} (C);
          \draw[->,thick] (F) --node[below] {} (E);
          \draw[->,thick] (G) --node[below] {} (F);
          \draw[->,thick] (H) --node[below] {$\left(F( \id_{b_{{\nindex}-j}};\rho_{b_j})\right)_{j \in {\nindex}}$} (G);

          \draw[->,thick] (C) --node[right] {$F\left( {\alpha_{\nindex}} \right)$} (G);
          \draw[->,thick, label={[mylabel]{{$\left(F( {\alpha_{{\nindex}-j}};\omega_j)\right)_{j \in {\nindex}}$}}  }] (D) -- (H);

          \draw[->,thick] (B) --node[left] {$\widehat{F}( \{ \alpha_{\nindex} \} )$} (F);

          \node at (1*\w,-0.5*\h) {$\circlearrowleft$};
          \node at (-0.5*\w,-0.5*\h) {$\circlearrowleft$};

\end{tikzpicture}
\end{center}
\end{construction}
\begin{proof}[Correctness of the construction]
      The diagram defining the action of $\widehat{F}$ on morphisms commutes
      by the functoriality of $F$ and by the equation $\rho_{a_j} \cdot \alpha_j = \omega_j \cdot \rho_{b_j}$.
      Moreover, $\widehat{F}$ defines a functor since taking cokernels of commutative squares is a functorial operation
      whose output does not depend on the morphism witnesses $\omega_i$.
\end{proof}

\begin{lemma}\label{lemma:extension_commutes_with_cokernels_mult}
      The functor $\widehat{F}$ described in \Cref{construction:extension_of_multi_functor}
      is multilinear and right exact.
\end{lemma}
\begin{proof}
      The multilinearity of $\widehat{F}$ follows from the multilinearity of $F$.
      In order to prove right exactness, we check the criterion stated in \Cref{lemma:criterion_resp_cok}.
      Let $(A_i \xrightarrow{\{\alpha_i\}} B_i)_{i \in \nindex}$ be a tuple of morphisms in $\prod_{i \in \nindex }\Freyd( \AC_i )$.
      By \Cref{construction:cokernels_in_freyd}, the cokernel of each $\{\alpha_i\}$ is given by
      \[C_i := (b_i \xleftarrow{\pmatcol{\rho_{b_i}}{\alpha_i}} r_{b_i} \oplus a_i) \, . \]
      Now, a diagram chase proves that the top row of the following commutative diagram with exact columns
      is exact, which yields the claim.
      \begin{center}
            \begin{tikzpicture}[label/.style={postaction={
                  decorate,
                  decoration={markings, mark=at position .5 with \node #1;}},
                  mylabel/.style={thick, draw=none, align=center, minimum width=0.5cm, minimum height=0.5cm,fill=white}}]
                  \coordinate (r) at (5.0,0);
                  \coordinate (rs) at (2.5,0);
                  \coordinate (d) at (0,-2);
                  \coordinate (ds) at (0,-1.5);
                  \node (A1) {};
                  \node (A2) at ($(A1) + (rs)$) {$0$};
                  \node (A3) at ($(A2) + (r)$) {$0$};
                  \node (A4) at ($(A3) + (r)$) {$0$};
                  
                  \node (B1) at ($(A1) + (ds)$) {$0$};
                  \node (B2) at ($(B1) + (rs)$) {$\widehat{F}( C_{\nindex} )$};
                  \node (B3) at ($(B2) + (r)$) {$\widehat{F}( B_{\nindex} ) $};
                  \node (B4) at ($(B3) + (r)$) {$\oplus_{j \in \nindex} \widehat{F}(B_{\nindex -j}; A_{j} )$};
                  
                  \node (C1) at ($(B1) + (d)$) {};
                  \node (C2) at ($(C1) + (rs)$) {$F(b_{\nindex})$};
                  \node (C3) at ($(C2) + (r)$) {$F(b_{\nindex})$};
                  \node (C4) at ($(C3) + (r)$) {$\oplus_{j \in \nindex} F(b_{\nindex -j}; a_{j})$};
                  
                  \node (D1) at ($(C1) + (d)$) {};
                  \node (D2) at ($(D1) + (rs)$) {$\oplus_{j \in \nindex} \left(F(b_{\nindex-j}; r_{b_{j}}) \oplus F(b_{\nindex -j}; a_{j})\right)$};
                  \node (D3) at ($(D2) + (r)$) {$\oplus_{j \in \nindex}F(b_{\nindex-j}; r_{b_{j}})$};
                  
                  \draw[->,thick] (D2) -- (C2);
                  \draw[->,thick] (C2) -- (B2);
                  \draw[->,thick] (B2) -- (A2);
                  \draw[->,thick] (D3) -- (D2);
                  \draw[->,thick] (C3) -- (C2);

                  \draw[->,thick] (B2) -- (B1);
                  \draw[->,thick] (B3) -- (B2);
                  \draw[->,thick] (B4) -- (B3);
                  
                  \draw[->,thick] (C4) -- (C3);
                  
                  \draw[->,thick] (C4) -- (B4);
                  \draw[->,thick] (B4) -- (A4);

                  \draw[->,thick] (D3) -- (C3);
                  \draw[->,thick] (C3) -- (B3);
                  \draw[->,thick] (B3) -- (A3);
                  \end{tikzpicture} \qedhere
      \end{center}
\end{proof}

\begin{construction}\label{construction:extension_of_multi_nat}
      Let $F,G: \prod_{i\in {\nindex}}\AC_i \longrightarrow \BC$ be multilinear functors and let
      \[
            \nu: F \longrightarrow G
      \]
      be a natural transformation.
      We construct the components of a natural transformation 
      \begin{center}
            \begin{tikzpicture}[baseline=(current bounding box.center),label/.style={postaction={
                      decorate,
                      decoration={markings, mark=at position .5 with \node #1;}},
                      mylabel/.style={thick, draw=black, align=center, minimum width=0.5cm, minimum height=0.5cm,fill=white}}]
            
                      \def\w{4.5};
                      \def\h{0.7};
            
                      \node (A) at (0,0) {$\prod_{i\in {\nindex}}\Freyd( \AC_i )$};
                      \node (B) at (\w,0) {$\BC$};
                     
                      \node (D) at (0.5*\w,0) [rotate = -90] {$\Rightarrow$};
                      \node (E) at (0.6*\w,0) {$\widehat{\nu}$};
                      
                      \draw[bend left,->,thick,label={[above]{$\widehat{F}$}},out=360+25,in=360+155] (A) to (B);
                      \draw[bend right,->,thick,label={[below]{$\widehat{G}$}},out=360-25,in=360-155] (A) to (B);
                     
            \end{tikzpicture}
            \end{center}
      between $\widehat{F}$ and $\widehat{G}$ as the unique morphism which turns the following diagram into a commutative diagram with exact rows:
      \begin{center}
            \begin{tikzpicture}[baseline=(current bounding box.center),label/.style={postaction={
                      decorate,
                      decoration={markings, mark=at position .5 with \node #1;}},
                      mylabel/.style={thick, draw=none, align=center, minimum width=0.5cm, minimum height=0.5cm,fill=white}}]
            
                      \def\w{3.5};
                      \def\h{3};
            
                      \node (A) at (-2*\w,0) {$0$};
                      \node (B) at (-\w,0) {$\widehat{F} \big( A_{{\nindex}} \big)$};
                      \node (C) at (0,0) {$F(a_{{\nindex}})$};
                      \node (D) at (2*\w,0) {$\bigoplus_{j \in {\nindex}} F( a_{{\nindex}-j}; r_{a_j} )$};
                      \node (E) at (-2*\w,-\h) {$0$};
                      \node (F) at (-\w,-\h) {$\widehat{G} \big( A_{{\nindex}} \big)$};
                      \node (G) at (0,-\h) {$G(a_{{\nindex}})$};
                      \node (H) at (2*\w,-\h) {$\bigoplus_{j \in {\nindex}} G( a_{{\nindex}-j}; r_{a_j} )$};
            
                      \draw[->,thick] (B) --node[below] {} (A);
                      \draw[->,thick] (C) --node[below] {} (B);
                      \draw[->,thick] (D) --node[above] {$\left(F( \id_{a_{{\nindex}-j}};\rho_{a_j})\right)_{j \in {\nindex}}$} (C);
                      \draw[->,thick] (F) --node[below] {} (E);
                      \draw[->,thick] (G) --node[below] {} (F);
                      \draw[->,thick] (H) --node[below] {$\left(G( \id_{a_{{\nindex}-j}};\rho_{a_j})\right)_{j \in {\nindex}}$} (G);
            
                      \draw[->,thick] (C) --node[right] {$\nu_{a_{{\nindex}}}$} (G);
                      \draw[->,thick, label={[mylabel]{{$\left(\nu_{(a_{{\nindex}-j}; r_{a_j} )}  \right)_{j \in {\nindex}}$}}  }] (D) -- (H);
            
                      \draw[->,thick] (B) --node[left] {$\widehat{\nu}_{A_{{\nindex}}}$} (F);
            
                      \node at (1*\w,-0.5*\h) {$\circlearrowleft$};
                      \node at (-0.5*\w,-0.5*\h) {$\circlearrowleft$};
            
            \end{tikzpicture}
            \end{center}
\end{construction}
\begin{proof}[Correctness of the construction]
      The rectangle on the right hand side of the diagram defining $\widehat{\nu}$ commutes since $\nu$ is a natural transformation.
      To show that the above data define a natural transformation, we have to verify that the following diagram commutes
      for all $(A_i \xrightarrow{\{\alpha_i\}} B_i)_{i \in {\nindex}}$:
      \begin{center}
      \begin{tikzpicture}[baseline=(current bounding box.center),label/.style={postaction={
                decorate,
                decoration={markings, mark=at position .5 with \node #1;}},
                mylabel/.style={thick, draw=black, align=center, minimum width=0.5cm, minimum height=0.5cm,fill=white}}]
      
                \def\w{4};
                \def\h{2.5};
      
                \node (A) at (0,0) {$\widehat{F} \big( A_{{\nindex}} \big)$};
                \node (B) at (\w,0) {$\widehat{G} \big( A_{{\nindex}} \big)$};
                \node (C) at (0,-\h) {$\widehat{F} \big( B_{{\nindex}} \big)$};
                \node (D) at (\w,-\h) {$\widehat{G} \big( B_{{\nindex}} \big)$};
      
                \draw[->,thick] (A) --node[below]{$\widehat{\nu}_{A_{{\nindex}}}$} (B);
                \draw[->,thick] (C) --node[above]{$\widehat{\nu}_{B_{{\nindex}}}$} (D);
                \draw[->,thick] (A) --node[left]{$\widehat{F}( \{\alpha_{{\nindex}}\} )$} (C);
                \draw[->,thick] (B) --node[right]{$\widehat{G}( \{\alpha_{{\nindex}}\} )$} (D);
      
      \end{tikzpicture}
      \end{center}
      But this diagram commutes because it fits into a diagram of the form
      \begin{center}
            \begin{tikzpicture}[baseline=(current bounding box.center),label/.style={postaction={
                      decorate,
                      decoration={markings, mark=at position .5 with \node #1;}},
                      mylabel/.style={thick, draw=black, align=center, minimum width=0.5cm, minimum height=0.5cm,fill=white}}]
            
                      \def\w{12};
                      \def\h{4};
                      \coordinate (r) at (4,0);
                      \coordinate (d) at (0,-1);
                      \node (A) at (0,0) {$F(a_{{\nindex}})$};
                      \node (B) at (\w,0) {$G(a_{{\nindex}})$};
                      \node (C) at (0,-\h) {$F(b_{{\nindex}})$};
                      \node (D) at (\w,-\h) {$G(b_{{\nindex}})$};
                      
                      \node (A2) at ($(A) + (r) + (d)$) {$\widehat{F} \big( A_{{\nindex}} \big)$};
                      \node (B2) at ($(B) - (r) + (d)$) {$\widehat{G} \big( A_{{\nindex}} \big)$};
                      \node (C2) at ($(C) + (r) - (d)$) {$\widehat{F} \big( B_{{\nindex}} \big)$};
                      \node (D2) at ($(D) - (r) - (d)$) {$\widehat{G} \big( B_{{\nindex}} \big)$};

                      \draw[->,thick] (A) --node[above]{$\nu_{a_{{\nindex}}}$} (B);
                      \draw[->,thick] (C) --node[below]{$\nu_{b_{{\nindex}}}$} (D);
                      \draw[->,thick] (A) --node[left]{$F ({\alpha_{{\nindex}}})$} (C);
                      \draw[->,thick] (B) --node[right]{$G ({\alpha_{{\nindex}}})$} (D);
                  
                      \draw[->,thick] (A2) --node[below]{$\widehat{\nu}_{A_{{\nindex}}}$} (B2);
                      \draw[->,thick] (C2) --node[above]{$\widehat{\nu}_{B_{{\nindex}}}$} (D2);
                      \draw[->,thick] (A2) --node[left]{$\widehat{F}( \{\alpha_{{\nindex}}\} )$} (C2);
                      \draw[->,thick] (B2) --node[right]{$\widehat{G}( \{\alpha_{{\nindex}}\} )$} (D2);
                      
                      \draw[->>,thick] (A) -- (A2);
                      \draw[->>,thick] (B) -- (B2);
                      \draw[->>,thick] (C) -- (C2);
                      \draw[->>,thick] (D) -- (D2);
            \end{tikzpicture}
            \end{center}
      in which all rectangles that involve outer nodes commute, and in which $F(a_{\nindex}) \rightarrow \widehat{F}( A_{\nindex})$ is the epimorphism induced by \Cref{construction:extension_of_multi_functor}.
\end{proof}

By the functoriality of the cokernel, it is easy to see that
\[
      \widehat{\id_F} = \id_{\widehat{F}}
\]
for all multilinear functors $F: \prod_{i \in {\nindex}}\AC_i \longrightarrow \BC$
and
\[
      \widehat{\nu \circ \mu} = \widehat{\nu} \circ \widehat{\mu}
\]
for all composable natural transformations $\nu$ and $\mu$ between multilinear functors from $\prod_{i \in {\nindex}}\AC_i$ to $\BC$.
Thus, due to \Cref{construction:extension_of_multi_functor} and \Cref{construction:extension_of_multi_nat},
and \Cref{lemma:extension_commutes_with_cokernels_mult} we have a well-defined functor
\[
\Homlin ( (\AC_i)_{i \in {\nindex}}, \BC ) \longrightarrow \Homlincok ( (\Freyd( \AC_i ) )_{i \in {\nindex}}, \BC ): F \mapsto \widehat{F}
\]
that will turn out to constitute one direction of the desired equivalence stated in \Cref{theorem:2_cat_universal_freyd_mult}.
The other direction will simply be given by restriction:
\[
 \Homlincok ( (\Freyd( \AC_i ) )_{i \in {\nindex}}, \BC ) \longrightarrow \Homlin ( (\AC_i)_{i \in {\nindex}}, \BC ): G \mapsto G|_{\AC_{\nindex}} := G \circ \emb
\]
where $\emb: \prod_{i \in {\nindex}} \AC_i \hookrightarrow \prod_{i \in {\nindex}} \Freyd( \AC_i )$
denotes the componentwise embedding.

\begin{lemma}\label{lemma:natural_isos}
      There is an isomorphism
      \[
            F \simeq \widehat{F}|_{\AC_{\nindex}}
      \]
      natural in $F \in \Homlin ( (\AC_i)_{i \in {\nindex}}, \BC )$.
      Moreover, there is an isomorphism
      \[
            G \simeq \widehat{G|_{\AC_{\nindex}}}
      \]
      natural in $G \in \Homlincok ( (\Freyd( \AC_i ) )_{i \in {\nindex}}, \BC )$.
\end{lemma}
\begin{proof}
      For $a_{\nindex} \in \prod_{i \in {\nindex}}\AC_i$, we obtain the desired natural isomorphism as
      \begin{align*}
            \widehat{F}(\emb( a_{\nindex} ) ) &\simeq \cokernel\left( F(a_{\nindex}) \longleftarrow \bigoplus_{j \in {\nindex}}F(a_{{\nindex}-j}; 0) \right) \\
            &\simeq \cokernel\left( F(a_{\nindex}) \longleftarrow 0 \right)  \\
            &\simeq F( a_{\nindex} ) \, .
      \end{align*}
      For $A_{\nindex} \in \prod_{i \in {\nindex}}\Freyd( \AC_i )$, we obtain the desired natural isomorphism as
      \begin{align*}
            G( A_{\nindex} ) &\simeq \cokernel\left( G(\emb(a_{\nindex})) \longleftarrow \bigoplus_{j \in {\nindex}}G(\emb(a_{{\nindex}-j};r_{a_j})) \right) & \text{(\Cref{lemma:criterion_resp_cok})}\\
            &\simeq \cokernel\left( {G|_{\AC_{\nindex}}}(a_{\nindex}) \longleftarrow \bigoplus_{j \in {\nindex}}{G|_{\AC_{\nindex}}}(a_{{\nindex}-j}; r_{a_j}) \right) \\
            &\simeq \widehat{G|_{\AC_{\nindex}}}( A_{\nindex} ) \, . & \qedhere
      \end{align*}
\end{proof}

In particular, \Cref{lemma:natural_isos}
states that the functors $F \mapsto \widehat{F}$ and $G \mapsto G|_{\AC_I}$
define an equivalence of categories, and thus, \Cref{theorem:2_cat_universal_freyd_mult} holds.

\section{Right exact monoidal structures on Freyd categories} \label{sec:FPMonoidalStructuresOnFreyd}

\subsection{Monoidal structures on Freyd categories} \label{subsec:MonoidalStructures}

In this Subsection, we recall the standard notions of the theory of monoidal categories,
but we will state the definitions within the context of Freyd categories.
The reason is that this will simplify our explanation of \fp promonoidal categories in the subsequent Subsection \ref{subsection:defining_pre_structures}.

\begin{notationnonumber}
      In this paper we frequently extend multilinear functors from $\prod_{i \in I}{\AC_i}$ to $\prod_{i \in I}{\Freyd( \AC_i )}$ and usually we reserve the symbol $\widehat{\phantom{-}}$ to denote such an extension. 
      However, in this very Subsection \ref{subsec:MonoidalStructures}, we make an exception of this rule,
      and use symbols like $\That$, $\Asshat$, $\Brhat$
      as variables that do not necessarily refer to such an extension.
      Again, the reason is a simplification of our explanation of \fp promonoidal categories in the subsequent Subsection.
\end{notationnonumber}

\subsubsection{Monoidal structures}\label{subsection:defining_monoidal_structures}

\begin{definition}[Semimonoidal structure] \label{def:SemiMonoidalStructure}
A semimonoidal structure on $\Freyd( \AC )$ consist of the following data:
\begin{enumerate}
 \item A bilinear functor $\oF \colon \Freyd( \AC ) \times \Freyd( \AC ) \to \Freyd( \AC )$ (\emph{tensor product}).
 \item An isomorphism
      $\Asshat_{A,B,C} \colon A \oF( B \oF C ) \xrightarrow{\sim} ( A \oF B ) \oF C $ natural in $A,B,C \in \Freyd( \AC )$ (\emph{associator}).
\end{enumerate}
This data is subject to the condition that the following diagram commutes for all $A,B,C,D \in \Freyd( \AC )$ (\emph{pentagon identity}):
\begin{equation} \label{diag:PentagonalIdentity}
\begin{tikzpicture}[baseline=(current bounding box.center),label/.style={postaction={
                  decorate,
                  decoration={markings, mark=at position .5 with \node #1;}},
                  mylabel/.style={thick, draw=black, align=center, minimum width=0.5cm, minimum height=0.5cm,fill=white}}]
                  
                \def\w{2.5};
                \node (a) at (0,2*\w) {$A \oF( B \oF (C \oF D) )$};
                \node (bl) at (1.5*\w,\w) {$(A \oF B) \oF (C \oF D)$};
                \node (br) at (-1.5*\w,\w) {$A \oF ( (B \oF C) \oF D )$};
                \node (cl) at (1.5*\w,0) {$((A \oF B) \oF C) \oF D$};
                \node (cr) at (-1.5*\w,0) {$(A \oF (B \oF C)) \oF D$};
                \draw[->, thick] (cr) to node[pos=0.45, above] {$\Asshat_{A,B,C} \oF D$} (cl);
                \draw[->, thick] (br) to node[pos=0.45, left] {$\Asshat_{A, B \oF C, D}$} (cr);
                \draw[->, thick] (bl) to node[pos=0.45, right] {$\Asshat_{A \oF B, C, D}$} (cl);
                \draw[->, thick] (a) to node[pos=0.45, above left] {$A \oF \Asshat_{B,C,D}$} (br);
                \draw[->, thick] (a) to node[pos=0.45, above right] {$\Asshat_{A,B,C \oF D}$} (bl);
        
                \node at (0.0*\w,1*\w) {$\circlearrowleft$};
\end{tikzpicture}
\end{equation}
\end{definition}

\begin{definition}[Monoidal structure] \label{def:MonoidalStructure}
A monoidal structure on $\Freyd( \AC )$ is a semimonoidal structure on $\Freyd( \AC )$ together with
\begin{enumerate}
 \item an object $1 \in \Freyd( \AC )$ (\emph{tensor unit}),
 \item a natural isomorphism $\LUhat_A \colon 1 \oF A \xrightarrow{\sim} A$ (\emph{left unitor}),
 \item and a natural isomorphism $\RUhat_A \colon A \oF 1 \xrightarrow{\sim} A$ (\emph{right unitor})
\end{enumerate}
such that the following diagram commutes for all $A,B \in \Freyd( \AC )$ (\emph{triangle identity}):
\begin{equation} \label{diag:TriangleIdentity}
\begin{tikzpicture}[baseline=(current bounding box.center),label/.style={postaction={
              decorate,
              decoration={markings, mark=at position .5 with \node #1;}},
              mylabel/.style={thick, draw=black, align=center, minimum width=0.5cm, minimum height=0.5cm,fill=white}}]
  
          \def\w{4};
          \def\h{2.5};
  
          \node (A) at (0,0) {$A \oF (1 \oF B)$};
          \node (B) at (2*\w,0) {$( A \oF 1) \oF B$};
          \node (C) at (\w,-\h) {$A \oF B$};
  
          \draw[->,thick] (A) --node[above]{$\Asshat_{A,1,B}$} (B);
          \draw[->,thick] (A) --node[below left]{$\id_A \oF \LUhat_B$} (C);
          \draw[->,thick] (B) --node[below right]{$\RUhat_A \oF \id_B$} (C);
  
          \node at (1.0*\w,-0.45*\h) {$\circlearrowleft$};
  
\end{tikzpicture}
\end{equation}
\end{definition}

\subsubsection{Braided monoidal structures}

\begin{definition}[Braided monoidal structure]
A monoidal structure on $\Freyd( \AC )$ together with a natural isomorphism $\Brhat_{A,B}: A \oF B \xrightarrow{\sim} B \oF A$ (\emph{braiding}) for $A,B \in \Freyd( \AC )$ is called a \emph{braided monoidal structure} if the following diagrams commute
for all $A,B,C \in \Freyd( \AC )$:
\begin{itemize}
 \item compatibility with left- and right unitor
              \begin{equation} \label{diag:BraidingI}
              \begin{tikzpicture}[baseline=(current bounding box.center),label/.style={postaction={
              decorate,
              decoration={markings, mark=at position .5 with \node #1;}},
              mylabel/.style={thick, draw=black, align=center, minimum width=0.5cm, minimum height=0.5cm,fill=white}}]
              
              \def\w{4.5};
              \def\h{2};
           
              \node (at1) at (0,0) {$A \oF 1$};
              \node (1ta) at (\w,0) {$1 \oF A$};
              \node (a) at (0.5*\w,-\h) {$A$};
           
              \draw[->,thick] (at1) --node[above]{$\Brhat_{A,1}$} (1ta);
              \draw[->,thick] (1ta) --node[below right]{$\LUhat_A$} (a);
              \draw[->,thick] (at1) --node[below left]{$\RUhat_A$} (a);
           
              \node at (0.5*\w,-0.45*\h) {$\circlearrowleft$};
           
              \end{tikzpicture}
              \end{equation}
 \item hexagonal identity I:
              \begin{equation} \label{diag:BraidingII}
              \begin{tikzpicture}[baseline=(current bounding box.center),label/.style={postaction={
              decorate,
              decoration={markings, mark=at position .5 with \node #1;}},
              mylabel/.style={thick, draw=black, align=center, minimum width=0.5cm, minimum height=0.5cm,fill=white}}]
              
              \def\w{2.5};
              \def\h{2};
              \node (p1) at (-\w,2*\h) {$( A \oF B ) \oF C$};
              \node (p2) at (\w,2*\h) {$C \oF ( A \oF B )$};
              \node (p3) at (-2*\w,\h) {$A \oF ( B \oF C )$};
              \node (p4) at (2*\w,\h) {$( C \oF A ) \oF B $};
              \node (p5) at (-\w,0) {$A \oF ( C \oF B )$};
              \node (p6) at (\w,0) {$( A \oF C ) \oF B$};
              \draw[->,thick] (p1) to node[above] {$\Brhat_{A \oF B, C}$} (p2);
              \draw[->,thick] (p2) to node[above right] {$\Asshat_{C,A,B}$} (p4);
              \draw[->,thick] (p4) to node[below right] {$( \Brhat_{C,A} \oF \mathrm{id}_{B} )$} (p6);
              \draw[->,thick] (p1) to node[above left] {$\Asshat^{-1}_{A,B,C}$} (p3);
              \draw[->,thick] (p3) to node[below left] {$( \mathrm{id}_{A} \oF \Brhat_{B,C} )$} (p5);
              \draw[->,thick] (p5) to node[above] {$\Asshat_{A,C,B}$} (p6);
           
              \node at (-0.0*\w,1.0*\h) {$\circlearrowleft$};
           
              \end{tikzpicture}
              \end{equation}
 \item hexagonal identity II:
              \begin{equation} \label{diag:BraidingIII}
              \begin{tikzpicture}[baseline=(current bounding box.center),label/.style={postaction={
              decorate,
              decoration={markings, mark=at position .5 with \node #1;}},
              mylabel/.style={thick, draw=black, align=center, minimum width=0.5cm, minimum height=0.5cm,fill=white}}]
              
              \def\w{2.5};
              \def\h{2};
              \node (p1) at (-\w,2*\h) {$A \oF ( B \oF C )$};
              \node (p2) at (\w,2*\h) {$( B \oF C ) \oF A$};
              \node (p3) at (-2*\w,\h) {$(A \oF B ) \oF C$};
              \node (p4) at (2*\w,\h) {$B \oF ( C \oF A)$};
              \node (p5) at (-\w,0) {$( B \oF A ) \oF C$};
              \node (p6) at (\w,0) {$B \oF ( A \oF C)$};
              \draw[->,thick] (p1) to node[above] {$\Brhat_{A, B \oF C}$} (p2);
              \draw[->,thick] (p2) to node[above right] {$\Asshat^{-1}_{B,C,A}$} (p4);
              \draw[->,thick] (p4) to node[below right] {$\mathrm{id}_{B} \oF \Brhat_{C,A}$} (p6);
              \draw[->,thick] (p1) to node[above left] {$\Asshat_{A,B,C}$} (p3);
              \draw[->,thick] (p3) to node[below left] {$\Brhat_{A,B} \oF \mathrm{id}_{C}$} (p5);
              \draw[->,thick] (p5) to node[above] {$\Asshat^{-1}_{B,A,C}$} (p6);
           
              \node at (-0.0*\w,1.0*\h) {$\circlearrowleft$};
           
              \end{tikzpicture}
              \end{equation}
\end{itemize}
\end{definition}

\begin{definition}[Symmetric monoidal structure]
A braided monoidal structure on $\Freyd( \AC )$ with the property $\Brhat_{A,B} = \Brhat_{B,A}^{-1}$ for $A,B \in \Freyd( \AC )$ is called a \emph{symmetric monoidal structure}.
\end{definition}

\subsubsection{Internal hom structures}

\begin{definition}[Internal hom structure] \label{Def:InternalHom}
An internal hom structure on $\Freyd( \AC )$ is a monoidal structure together with the following data:
\begin{enumerate}
 \item For every $A \in \Freyd( \AC )$, an additive functor $\widehat{\underline{\mathrm{Hom}}}( A, - ): \Freyd( \AC ) \rightarrow \Freyd( \AC )$  (\emph{internal hom}).
 \item For every $B \in \AC$, a morphism $\PreCoevhat_{A,B}: A \longrightarrow \widehat{\underline{\mathrm{Hom}}}( B, A \oF B )$
       natural in $A \in \Freyd( \AC )$ (\emph{coevaluation}).
 \item For every $A \in \AC$, a morphism $\PreEvhat_{A,B}: \widehat{\underline{\mathrm{Hom}}}( A, B ) \oF A \longrightarrow B$ natural in $B \in \Freyd( \AC )$ (\emph{evaluation}).
\end{enumerate}
This data is subject to the condition that
\begin{equation} \label{diag:IntHomI}
      \begin{tikzpicture}[baseline=(current bounding box.center),label/.style={postaction={
      decorate,
      decoration={markings, mark=at position .5 with \node #1;}},
      mylabel/.style={thick, draw=none, align=center, minimum width=0.5cm, minimum height=0.5cm,fill=white}}]
      \coordinate (r) at (5.5,0);
      \coordinate (d) at (0,-3);
      \node (A) {$A \oF B$};
      \node (B) at ($(A) + 1.2*(r)$) {$\widehat{\underline{\mathrm{Hom}}}( B, A \oF B ) \oF B$};
      \node (C) at ($(B)+(d)$) {$A \oF B$};
      
      \draw[->,thick] (A) --node[above]{$\PreCoevhat_{A,B} \oF B$} (B);
      \draw[->,thick] (B) --node[right]{$\PreEvhat_{B, A \oF B}$} (C);
      \draw[->,thick] (A) --node[below left]{$\id$} (C);

\end{tikzpicture}
\end{equation}
and
\begin{equation} \label{diag:IntHomII}
      \begin{tikzpicture}[baseline=(current bounding box.center),label/.style={postaction={
      decorate,
      decoration={markings, mark=at position .5 with \node #1;}},
      mylabel/.style={thick, draw=none, align=center, minimum width=0.5cm, minimum height=0.5cm,fill=white}}]

      \coordinate (r) at (5.5,0);
      \coordinate (d) at (0,-3);

      \node (A2) {$\widehat{\underline{\mathrm{Hom}}}( A,B )$};
      \node (B2) at ($(A2) + 1.2*(r)$) {$\widehat{\underline{\mathrm{Hom}}}\left( A, \widehat{\underline{\mathrm{Hom}}}(A,B) \oF A\right)$};
      \node (C2) at ($(B2)+(d)$) {$\widehat{\underline{\mathrm{Hom}}}(A,B)$};
      
      \draw[->,thick] (A2) --node[above]{$\PreCoevhat_{\widehat{\underline{\mathrm{Hom}}}(A,B),A}$} (B2);
      \draw[->,thick] (B2) --node[right]{$\widehat{\underline{\mathrm{Hom}}}( A, \PreEvhat_{A, B} )$} (C2);
      \draw[->,thick] (A2) --node[below left]{$\id$} (C2);

      \end{tikzpicture}
  \end{equation}
commute for all $A,B \in \Freyd( \AC )$.
\end{definition}

\begin{remark}
      The conditions stated in \Cref{Def:InternalHom}
      are the triangle identities characterizing $\PreCoevhat$ and $\PreEvhat$
      as unit and counit of an adjunction $- \oF B \dashv \widehat{\underline{\mathrm{Hom}}}(B,-)$.
\end{remark}

\begin{remark}
      In this paper, we restrict our attention to right adjoints of the functors $- \oF B$,
      since the discussion of right adjoints of $B \oF -$ is similar.
\end{remark}

\subsection{\texorpdfstring{\Fp}{F.p.} promonoidal structures on additive categories} \label{subsection:defining_pre_structures}

In this Subsection, we define our notion of \fp promonoidal structures on additive categories $\AC$. In 
the subsequent Subsection \ref{subsec:LiftsOfPreMonoidalStructures} we will show that they lift to right exact monoidal structures on $\Freyd( \AC )$.

\begin{notationnonumber}
      Let us emphasize that for the rest of this paper we use $\widehat{\phantom{-}}$ to denote extensions of multilinear functors and natural transformations from $\prod_{i \in \nindex}{\AC_i}$ to $\prod_{i \in \nindex}{\Freyd( \AC )}$,
      as they were introduced in \Cref{subsection:universal_prop_freyd}.
      For a given bilinear functor $T(-,-)$ of any kind, we also use the infix notation $(- \otimes_T -)$.
\end{notationnonumber}

\subsubsection{\texorpdfstring{\Fp}{F.p.} promonoidal structures}

\begin{definition}[\Fp prosemimonoidal structure] \label{def:proSemiMonoidalStructure}
An \fp prosemimonoidal structure on $\AC$ consist of the following data:
\begin{enumerate}
 \item A bilinear functor $T \colon \AC \times \AC \to \Freyd( \AC )$ (\emph{\fp protensor product}).
 \item An isomorphism $\Ass_{a,b,c} \colon a \oF \left( b \oF c \right) \xrightarrow{\sim} \left( a \oF b \right) \oF c$ natural in $a,b,c \in \AC$ (\emph{\fp proassociator}).
\end{enumerate}
This data is subject to the condition that the following diagram commutes for all $a,b,c,d \in \AC$
\emph{(restricted pentagon identity)}:
\begin{equation} \label{diag:RestrictedPentagonalIdentity}
\begin{tikzpicture}[baseline=(current bounding box.center),label/.style={postaction={
                  decorate,
                  decoration={markings, mark=at position .5 with \node #1;}},
                  mylabel/.style={thick, draw=black, align=center, minimum width=0.5cm, minimum height=0.5cm,fill=white}}]
                  
                \def\w{2.5};
                \node (a) at (0,2*\w) {$a \oF( b \oF (c \oF d) )$};
                \node (bl) at (1.5*\w,\w) {$(a \oF b) \oF (c \oF d)$};
                \node (br) at (-1.5*\w,\w) {$a \oF ( (b \oF c) \oF d )$};
                \node (cl) at (1.5*\w,0) {$((a \oF b) \oF c) \oF d$};
                \node (cr) at (-1.5*\w,0) {$(a \oF (b \oF c)) \oF d$};
                \draw[->, thick] (cr) to node[pos=0.45, above] {$\Asshat_{a,b,c} \oF d$} (cl);
                \draw[->, thick] (br) to node[pos=0.45, left] {$\Asshat_{a, b \oF c, d}$} (cr);
                \draw[->, thick] (bl) to node[pos=0.45, right] {$\Asshat_{a \oF b, c, d}$} (cl);
                \draw[->, thick] (a) to node[pos=0.45, above left] {$a \oF \Asshat_{b,c,d}$} (br);
                \draw[->, thick] (a) to node[pos=0.45, above right] {$\Asshat_{a,b,c \oF d}$} (bl);
        
                \node at (0.0*\w,1*\w) {$\circlearrowleft$};
\end{tikzpicture}
\end{equation}
\end{definition}

\begin{remark}
      Note that in order to write down the restricted pentagon identity,
      it is necessary to use the extended functor $\That$, since $T(a,b)$ for $a,b \in \AC$
      does not have to be an object in $\AC$ (regarded as a full subcategory of $\Freyd( \AC )$, \Cref{remark:implicit_emb}).
      Moreover, let us point out the similarity between \Cref{diag:PentagonalIdentity} and \Cref{diag:RestrictedPentagonalIdentity}. 
      Namely, the latter is obtained by evaluating the former only at objects of $\AC$.
      This is why we term \Cref{diag:RestrictedPentagonalIdentity} the restriction of \Cref{diag:PentagonalIdentity} to $\AC$,
      and in a similar way, we will refer to restrictions of the triangle identity, hexagonal identities, etc.
\end{remark}

\begin{definition}[\Fp promonoidal structure] \label{def:preMonoidalStructure}
An \fp promonoidal structure on $\AC$ is an \fp prosemimonoidal structure on $\AC$ together with
\begin{enumerate}
 \item an object $1 \in \Freyd( \AC )$ (\emph{\fp protensor unit}),
 \item an isomorphism $\LU_a \colon 1 \oF  a  \xrightarrow{\sim} a $ natural in $a \in \AC$ (\emph{\fp left prounitor}),
 \item an isomorphism $\RU_a \colon a \oF 1 \xrightarrow{\sim} a$ natural in $a \in \AC$ (\emph{\fp right prounitor}),
\end{enumerate}
such that the restriction of the triangle identity \Cref{diag:TriangleIdentity} to $\AC$ commutes.
\end{definition}

\subsubsection{\Fp braided promonoidal structures}

\begin{definition}[\Fp braided promonoidal structure]
An \fp promonoidal structure on an additive category $\AC$ together with a natural isomorphism $\Br_{a,b}: a \oF b \xrightarrow{{\sim}} b \oF a$ (\emph{\fp probraiding}) 
for $a,b \in \AC$ is called an \emph{\fp braided promonoidal structure} if the restrictions of \Cref{diag:BraidingI}, \Cref{diag:BraidingII} and \Cref{diag:BraidingIII} to $\AC$ commute.
\end{definition}

\begin{definition}[Symmetric \fp promonoidal structure]
An \fp braided promonoidal structure on an additive category $\AC$ with the property $\Br_{a,b} = \Br_{b,a}^{-1}$ for all $a,b \in \AC$ is called a \emph{symmetric \fp promonoidal structure}.
\end{definition}

\subsubsection{\Fp prointernal hom}

\begin{definition}[\Fp prointernal hom] \label{Def:PreInternalHom}

An \fp prointernal hom structure on $\Freyd( \AC )$ is an \fp promonoidal structure together with the following data: 
\begin{enumerate}
 \item For every $a \in \AC$, an additive functor $\underline{\mathrm{Hom}}( a, - ): \AC \rightarrow \Freyd( \AC )$ (\emph{\fp prointernal hom}).
 \item For every $a \in \AC$, a morphism $\PreCoev_{b,a}: b \longrightarrow \underline{\mathrm{Hom}}( a, b \oF a )$ natural in $b \in \AC$ (\emph{\fp procoevaluation}).
 \item For every $a \in \AC$, a morphism $\PreEv_{a,b}: \underline{\mathrm{Hom}}( a, b ) \oF a \longrightarrow b$ natural in $b \in \AC$ (\emph{\fp proevaluation}).
\end{enumerate}
As usual, the extensions of these data to the Freyd category $\Freyd( \AC )$ are denoted by $\widehat{\underline{\mathrm{Hom}}}( a, - )$, $\PreCoevhat_{-,a}$ and $\PreEvhat_{a,-}$, respectively. 
Baring this notation in mind, the data is moreover subject to the condition that the restrictions of \Cref{diag:IntHomI} and \Cref{diag:IntHomII} to $\AC$ commute.
\end{definition}

\subsection{From \texorpdfstring{\fp}{f.p.} promonoidal structures to right exact monoidal structures} \label{subsec:LiftsOfPreMonoidalStructures}

In this Subsection, we describe how \fp promonoidal structures on an additive category $\AC$ give rise to
right exact monoidal structures on $\Freyd( \AC )$.
We fix an \fp protensor product $T \colon \AC \times \AC \to \Freyd( \AC )$,
and, as usual, denote its extension to $\Freyd( \AC )$ by $\oF$.
Whenever we refer to another promonoidal datum, like an \fp proassociator,
we do this w.r.t.\ $T$.

The proofs on how promonoidal data on $\AC$ extend to monoidal data on $\Freyd( \AC )$
all follow the same scheme: we check the defining identities by restricting to $\AC$
and applying \Cref{theorem:2_cat_universal_freyd_mult}. Only the case of extending \fp pre internal homs 
to internal homs is more involved, which is due to its contravariant first component.

\subsubsection{Associators}\label{subsubsection:associators}

\begin{lemma}\label{lemma:Asshat_iso}
      Let $\Ass$ be an \fp proassociator.
      Then
      $\Asshat_{A,B,C}$ is an isomorphism for all $A,B,C \in \Freyd( \AC )$
      if and only if $\Ass_{a,b,c}$ is an isomorphism for all $a,b,c \in \Freyd( \AC )$.
\end{lemma}
\begin{proof}
      We set 
      \[
            F\colon  \Freyd( \AC )^3 \rightarrow \Freyd( \AC ): (A,B,C) \mapsto A \oF ( B \oF C )
      \]
      and
      \[
            G\colon  \Freyd( \AC )^3 \rightarrow \Freyd( \AC ): (A,B,C) \mapsto (A \oF  B) \oF C \, .
      \]
      Now, the claim follows since, by \Cref{theorem:2_cat_universal_freyd_mult}, we have a bijection between sets of natural transformations
      \begin{align*}
            \Hom_{\Homlincok( \Freyd( \AC )^3, \Freyd( \AC ) )}(F,G) &\xrightarrow{\sim} \Hom_{\Homlin( \AC^3, \Freyd( \AC ) )}(F \circ \emb, G \circ \emb) \\
            \nu &\mapsto \nu \circ \emb
      \end{align*}
      which respects isomorphisms.
\end{proof}

\begin{lemma}\label{lemma:pentagonal_id}
      Let $\Ass$ be an \fp proassociator (for $T$).
      Then $\Asshat$ satisfies the pentagonal identity if and only if the restricted pentagonal identity in \Cref{diag:RestrictedPentagonalIdentity} is satisfied.
\end{lemma}
\begin{proof}
      We set 
      \[
            F\colon - \oF (- \oF ( - \oF - ) ): \Freyd( \AC )^4 \rightarrow \Freyd( \AC )
      \]
      and
      \[
            G\colon ( ( - \oF -) \oF - ) \oF - : \Freyd( \AC )^4 \rightarrow \Freyd( \AC ) \, .
      \]
      We need to compare the two isomorphisms
      \[
            p_{A,B,C,D} := \Asshat_{A,B,C \oF D} \circ \Asshat_{A, B \oF C, D} \circ \left(A \oF \Asshat_{B,C,D}\right): F(A,B,C,D) \xrightarrow{{\sim}} G(A,B,C,D)
      \]
      and
      \[
            q_{A,B,C,D} := \Asshat_{A,B,C \oF D} \circ \Asshat_{A \oF B, C, D}: F(A,B,C,D) \xrightarrow{{\sim}} G(A,B,C,D)
      \]
      natural in $A,B,C,D \in \Freyd( \AC )$.
      By \Cref{theorem:2_cat_universal_freyd_mult}, the map
      \[
            \Hom_{\Homlincok(\Freyd( \AC )^4, \Freyd( \AC ) )}( F, G ) \rightarrow \Hom_{\Homlin( \AC^4, \Freyd( \AC ))}( F \circ \emb, G \circ \emb )
      \]
      is an isomorphism, and thus, $p \circ \emb = q \circ \emb$ if and only if $p = q$.
      But $p \circ \emb = q \circ \emb$ is exactly the restricted pentagonal identity.
\end{proof}

\subsubsection{Unitors}\label{subsubsection:unitors}
For the following two lemmata, let $1$ be an \fp protensor unit,
$\LU$ be an \fp left prounitor, $\RU$ be an \fp right prounitor.

\begin{lemma}
      $\LU_a$ is an isomorphism for all $a \in \AC$ if and only if $\LUhat_A$ is an isomorphism for all $A \in \Freyd( \AC )$.
      Similarly, 
      $\RU_a$ is an isomorphism for all $a \in \AC$ if and only if $\RUhat_A$ is an isomorphism for all $A \in \Freyd( \AC )$.
\end{lemma}

\begin{lemma}
      $\LUhat$ and $\RUhat$ satisfy the triangle identity \Cref{diag:TriangleIdentity} if and only if 
      the restriction of the triangle identity to $\AC$ is satisfied.
\end{lemma}

For both proofs, we may use the same proof strategy as in Subsubsection \ref{subsubsection:associators}.

\subsubsection{Braiding}
For the following two lemmata, let $\Br$ be an \fp probraiding.

\begin{lemma}
      $\Brhat_{A,B}$ is an isomorphism for all $A,B \in \Freyd( \AC )$ if and only if $\Br_{a,b}$ is an isomorphism for all $a,b \in \Freyd( \AC )$.
\end{lemma}

\begin{lemma}
$\Brhat_{-,-}$ defines a braided monoidal structure for the monoidal structure on $\Freyd( \AC )$ defined by $\oF$ and $\Asshat$ if and only if we have
\begin{itemize}
   \item $\Brhat$ restricted compatibility with left- and right unitor, i.e. commutativity of \Cref{diag:BraidingI} restricted to $\AC$,
   \item commutativity of restricted hexagonal identity I, i.e. of \Cref{diag:BraidingII} restricted to $\AC$,
   \item commutativity of restricted hexagonal identity II, i.e. of \Cref{diag:BraidingIII} restricted to $\AC$.
\end{itemize}
Moreover, $\Brhat$ is symmetric if and only if the identity
\[ \Br_{b,a} \circ \Br_{a,b} = \id_{a \oF b} \]
holds for all $a,b \in \AC$.
\end{lemma}

Again, for both proofs, we may use the same proof strategy as in Subsubsection \ref{subsubsection:associators}.

\subsubsection{Internal homomorphisms}

The goal of this Subsection is to find sufficient conditions
that allow us to create an internal hom functor for $\oF$.

\begin{lemma}
Suppose we are given an \fp prointernal hom structure on $\Freyd( \AC )$.
Then we have
\[ \left(- \oF a \right) \dashv \widehat{\mathrm{\underline{Hom}} ( a, - )} \]
for all $a \in \AC$, \ie, $- \oF a: \Freyd( \AC ) \rightarrow \Freyd( \AC )$ is left adjoint to $\widehat{\mathrm{\underline{Hom}}} ( a, - ): \Freyd( \AC ) \rightarrow \Freyd( \AC )$.
\end{lemma}
\begin{proof}
The unit and counit of the desired adjunction are given by
$\PreCoevhat_{-,a}$ and $\PreEvhat_{a,-}$,
since they satisfy \Cref{diag:IntHomI} and \Cref{diag:IntHomII}
if and only if the restrictions of these equations are satisfied.
\end{proof}

Thus, an \fp prointernal hom structure yields right adjoints for tensoring with objects $a \in \AC$.
The next theorem provides a sufficient condition for the existence of right adjoints for tensoring with
an arbitrary object $A \in \Freyd( \AC )$.

\begin{theorem}\label{theorem:homA}
Let $A = (a \xleftarrow{\rho_a} r_a) \in \Freyd( \AC )$, and suppose we are given right adjoints
\[
\left(- \oF a\right) \dashv \widehat{\mathrm{\underline{Hom}} ( a, - )}
\hspace{3em} \text{and} \hspace{3em}
\left(- \oF r_a\right) \dashv \widehat{\mathrm{\underline{Hom}} ( r_a, - )} \, . \]
If $\AC$ has weak kernels, then $\left(- \oF A\right)$ has a right adjoint,
which we denote by $\widehat{\mathrm{\underline{Hom}}} ( A, - )$:
\[ \left(- \oF A\right) \dashv \widehat{\mathrm{\underline{Hom}}} ( A, - ) \, . \]
\end{theorem}
\begin{proof}
      We proceed in three steps.
      First, we define a natural morphism
      \[
      \widehat{\mathrm{\underline{Hom}} ( a, - )} \xrightarrow{\widehat{\mathrm{\underline{Hom}}} ( \rho_a, - )} \widehat{\mathrm{\underline{Hom}}} ( r_a, - ) \, ,
      \]
      second, we show that setting
      \[
            \widehat{\mathrm{\underline{Hom}}} ( A, C ) := \kernel( \widehat{\mathrm{\underline{Hom}}} ( \rho_a, C )
      \]
      for $C \in \Freyd( \AC )$
      gives rise to a well-defined functor, and third, we prove the desired adjunction.
      
      For the first step, we take a look at the following diagram
      \begin{center}
            \begin{tikzpicture}[label/.style={postaction={
            decorate,
            decoration={markings, mark=at position .5 with \node #1;}},
            mylabel/.style={thick, draw=none, align=center, minimum width=0.5cm, minimum height=0.5cm,fill=white}}]
            \coordinate (r) at (9,0);
            \coordinate (d) at (0,-2);
            \node (A) {$\Hom( B, \widehat{\mathrm{\underline{Hom}}} ( a, C ) )$};
            \node (B) at ($(A) + (r)$) {$\Hom( B, \widehat{\mathrm{\underline{Hom}}} ( r_a, C ) )$};
            \node (C) at ($(A)+(d)$) {$\Hom( B \oF a, C )$};
            \node (D) at ($(B)+(d)$) {$\Hom( B \oF r_a, C )$};
            
            \draw[->,thick, dashed] (A) -- (B);
            \draw[->,thick] (C) --node[above]{$\Hom( B \oF \rho_a, C )$} (D);
            \draw[->,thick] (A) --node[above,rotate=90]{$\sim$} (C);
            \draw[->,thick] (D) --node[below,rotate=90]{$\sim$} (B);
            \end{tikzpicture}
        \end{center}
        for all $B, C \in \Freyd( \AC )$.
        The dashed arrow that renders this diagram commutative is uniquely determined.
        Furthermore, by Yoneda's lemma, it is of the form
        \[
              \Hom( B, \widehat{\mathrm{\underline{Hom}}} ( \rho_a, C ) )
        \]
        for some morphism $\widehat{\mathrm{\underline{Hom}}} ( \rho_a, C ) \colon \widehat{\mathrm{\underline{Hom}}} ( a, C ) \rightarrow \widehat{\mathrm{\underline{Hom}}} ( r_a, C )$,
        and its naturality in $C \in \Freyd( \AC )$ is implied by the naturality of the morphisms in the above diagram.
        
        For the second step, we note that $\AC$ having weak kernels implies $\Freyd( \AC )$ being abelian by \Cref{theorem:freyd},
        so in particular, it has kernels.
        Moreover, by the naturality of $\widehat{\mathrm{\underline{Hom}}} ( \rho_a, - )$,
        the operation $\kernel( \widehat{\mathrm{\underline{Hom}}} ( \rho_a, C ) )$
        is functorial.
        
        Last, we prove adjointness by the following chain of morphisms
        in which each step is natural in $B, C \in \Freyd( \AC )$:
        \begin{align*}
            \Hom( B \oF A, C ) &\simeq \Hom( B \oF \cokernel( \rho_a ), C ) & \text{($A \simeq \cokernel( \rho_a)$)} \\
            &\simeq \Hom( \cokernel( B \oF \rho_a ), C ) & \text{($\oF$ is right exact)} \\
            &\simeq \kernel (\Hom( B \oF \rho_a, C ) )  & \text{($\Hom$ is left exact)}\\
            &\simeq \kernel (\Hom( B , \widehat{\mathrm{\underline{Hom}}} ( \rho_a, C ) ) ) & \text{(Definition of $\widehat{\mathrm{\underline{Hom}}} ( \rho_a, - )$)}\\
            &\simeq  \Hom( B , \kernel(\widehat{\mathrm{\underline{Hom}}} ( \rho_a, C ) ) ) & \text{($\Hom$ is left exact)}\\
            &\simeq  \Hom( B , \widehat{\mathrm{\underline{Hom}}} ( A, C ) ) ) & \text{(Definition of $\widehat{\mathrm{\underline{Hom}}} ( A, - )$)} & \qedhere
        \end{align*}
\end{proof}

\section{Connection with Day convolution} \label{subsec:ConnectionToDay}
In \cite{Day70}, Day analyzes closed monoidal structures on 
functor categories in the context of enriched category theory.
In this section, we compare Day's findings to 
the tensor products on Freyd categories described in this paper,
since Freyd categories can be seen as subcategories of functor categories, see \Cref{theorem:freyd_as_fp_functors}.

\subsection{Introduction to Day convolution}\label{subsection:intro_day}
We start by shortly describing the theory of so-called Day convolutions in
the additive context, \ie, in the context of $\Ab$-enriched category theory.

Let $\AC$ be a small additive category.
For every given multilinear functor 
\[P: \AC^{\op} \times \AC^{\op} \times \AC \rightarrow \Ab \, ,\]
the \textbf{Day convolution} of two functors 
$F, G: \AC \rightarrow \Ab$
w.r.t.\ $P$ is given by the following coend\footnote{For a short introduction to the calculus of (co)ends, see, \eg, \cite{MLCWM}}:
\[ F \ast_{P} G \coloneqq \int^{ x,y \in \AC }{P(x,y,-) \otimes_{\Z} F( x ) \otimes_{\Z} G( y )} \, . \]
Since $\AC$ is small and $\Ab$ cocomplete,
this coend exists and defines a bilinear functor
\[
      - \ast_{P} -: \Hom( \AC, \Ab ) \times \Hom( \AC, \Ab ) \rightarrow \Hom( \AC, \Ab ) \, .
\]
Fixing the right component, the resulting univariate functor $(- \ast_{P} G)$ always admits a right adjoint
$(-/G)$, given by the formula
\[
      (F/G) :=  \int_{x \in \AC}{\Hom_{\Z}\left( \int^{ y \in \AC }{ Gy \otimes_{\Z} P(-,y,x ) }, Fx \right)} \, .
\]
Similarly, fixing the left component, the resulting univariate functor $(F \ast_{P} -)$ always admits a right adjoint
$(F\backslash-)$, given by the formula
\[
      (F\backslash G) :=  \int_{x \in \AC}{\Hom_{\Z}\left( \int^{ y \in \AC }{ Fy \otimes_{\Z} P(y,-,x ) }, Gx \right)} \, .
\]
These assertions are validated by the computations in the proof of Theorem $3.3$ in \cite{Day70}.

Since left adjoints commute with arbitrary colimits,
we may conclude that $(- \ast_{P} -)$ commutes in particular componentwise with cokernels.
Moreover, if we evaluate $(- \ast_{P} -)$ at representables $h^a := \Hom_{\AC}( a, - )$ and
$h^b := \Hom_{\AC}( b, -)$ for $a,b \in \AC$, we get
\begin{align*}
      h^a \ast_{P} h^b &\simeq
      \int^{ x,y \in \AC }{P(x,y,-) \otimes_{\Z} h^a( x ) \otimes_{\Z} h^b( y )} \\
      &\simeq
      \int^{ x \in \AC }{  h^a( x )\otimes_{\Z} \left( \int^{ y \in \AC }{P(x,y,-) \otimes_{\Z} h^b( y ) } \right) } \\
      &\simeq
      \int^{ x \in \AC }{  h^a( x )\otimes_{\Z} P(x,b,-) } \\
      &\simeq
      P(a,b,-) \, ,
\end{align*}
where we used the expansion of a double integral to an iterated integral and the co-Yoneda lemma\footnote{
      The co-Yoneda lemma states that any functor $F: \AC^{\op} \rightarrow \Ab$
      is equivalent to the coend $\int^{ x \in \AC }{  \Hom_{\AC}(-,x)\otimes_{\Z} F(x) }$,
      \ie, every presheaf is a particular colimit of representables.
}.
In other words, the functor $P$ prescribes the values of $(- \ast_{P} -)$
on representables up to natural isomorphism.

\subsection{Day convolution of finitely presented functors}

We are now ready to
formulate the link between Day convolution and tensor products on Freyd categories.
Let $T: \AC^{\op} \times \AC^{\op} \rightarrow \Freyd( \AC^{\op} )$
be a bilinear functor.
The functor $\oF$ induced by the universal property of Freyd categories
gives rise to a functor
\[
      \fpf( \AC, \Ab) \times \fpf( \AC, \Ab) \rightarrow \Hom( \AC, \Ab)
\]
where we use the equivalence $\fpf( \AC, \Ab) \simeq \Freyd( \AC^{\op} )$
and then postcompose with the inclusion $\fpf( \AC, \Ab) \hookrightarrow \Hom( \AC, \Ab)$.
By abuse of notation, we denote this functor again by $\oF$.
Moreover, $T$ also defines a functor
\[
      P: \AC^{\op} \times \AC^{\op} \times \AC \rightarrow \Ab: (a,b,c) \mapsto T(a,b)(c) \, ,
\]
where we interpret $T(a,b)$ as an object in $\fpf( \AC, \Ab)$ to which we may apply $c \in \AC$.
\begin{theorem}\label{theorem:restriction_day}
      The functor $(- \oF -)$ is equivalent to the restriction of $(-\ast_{P}-)$
      to finitely presented functors, \ie,
      the following diagram commutes up to natural isomorphism:
      \begin{center}
      \begin{tikzpicture}[baseline=(current bounding box.center),label/.style={postaction={
                decorate,
                decoration={markings, mark=at position .5 with \node #1;}},
                mylabel/.style={thick, draw=black, align=center, minimum width=0.5cm, minimum height=0.5cm,fill=white}}]
      
                \def\w{6};
                \def\h{3};
      
                \node (A) at (0,0) {$\mathrm{Hom} ( \AC, \Ab ) \times \mathrm{Hom} ( \AC, \Ab )$};
                \node (B) at (\w,0) {$\mathrm{Hom} ( \AC, \Ab )$};
                \node (C) at (0,-\h) {$\mathrm{fp} ( \AC, \Ab ) \times \mathrm{fp} ( \AC, \Ab )$};
                \node (D) at (\w,-\h) {$\mathrm{fp} ( \AC, \Ab )$};
      
                \draw[->,thick] (A) --node[above] {$\ast_{P}$} (B);
                \draw[right hook ->,thick] (C) -- (A);
                \draw[right hook ->,thick] (D) -- (B);
                \draw[->,thick] (C) --node[above] {$\oF$} (D);
      
                \node at (0.5*\w,-0.5*\h) {$\circlearrowleft$};
      
      \end{tikzpicture}
      \end{center}
      \end{theorem}
      
      \begin{proof}
            By Theorem \ref{theorem:2_cat_universal_freyd_mult},
            we have an equivalence of categories
      \begin{align*}
            \Homlin ( ( \AC^{\op}, \AC^{\op} ), \Hom( \AC, \Ab ) ) &\simeq \Homlincok \left( (\Freyd( \AC^{\op} ), \Freyd( \AC^{\op}) ), \Hom( \AC, \Ab ) \right) \\
            &\simeq \Homlincok \left( (\mathrm{fp} ( \AC, \Ab ), \mathrm{fp} ( \AC, \Ab ) ), \Hom( \AC, \Ab ) \right) \, .
      \end{align*}
      Since both paths in the given diagram are bilinear functors that commute with cokernels,
      they are naturally isomorphic
      if and only if their restrictions along the Yoneda embedding 
      \[
            \AC^{\op} \hookrightarrow \mathrm{fp} ( \AC, \Ab ) )
      \]
      are naturally isomorphic.
      But this is the case since 
      \[
            h^a \ast_P h^b \simeq P(a,b,-) \simeq T(a,b)(-) \simeq h^a \oF h^b
      \]
      as we have argued in Subsection \ref{subsection:intro_day}.
      \end{proof}
      
\begin{remark}
      In order to obtain a monoidal structure on
      $\Hom( \AC, \Ab )$, Day defines in \cite{Day70}
      premonoidal categories, a concept that he later renamed as
      \textbf{promonoidal categories} in \cite{Day74}, where he also reformulated these categories in the language of profunctors.
      A promonoidal category $\AC$ comes equipped with a multilinear functor 
      $P: \AC^{\op} \times \AC^{\op} \times \AC \rightarrow \Ab$,
      that, as we have seen, gives rise to a tensor product $(- \ast_{P} -)$ on $\Hom( \AC, \Ab )$.
      In order to have an associator for this tensor product,
      Day moreover requires a promonoidal category to be equipped with
      an isomorphism
      \[
            \int^{x \in \AC}{P(a,b,x) \otimes_{\Z} P(x,c,d)} \stackrel{\sim}{\longrightarrow}
            \int^{x \in \AC}{P(b,c,x) \otimes_{\Z} P(a,x,d)}
      \]
      natural in $a,b,c \in \AC^{\op}$ and $d \in \AC$
      satisfying a certain coherence relation
      that ensures this datum to correspond to an associator of $(- \ast_{P} -)$.
      We may rewrite the left hand side of this natural isomorphism as
      \begin{align*}
            \int^{x \in \AC}{P(a,b,x) \otimes_{\Z} P(x,c,d)}
            &\simeq \int^{x \in \AC}{P(a,b,x) \otimes_{\Z} \left(\int^{y \in \AC} h^c(y) \otimes_{\Z} P(x,y,d) \right)} \\
            &\simeq \int^{x,y \in \AC}{P(x,y,d) \otimes_{\Z} P(a,b,x) \otimes_{\Z} h^c(y) } \\
            &\simeq \left(P(a,b,-) \ast_{P} h^c\right)(d) \\
            &\simeq \left((h^a \ast_{P} h^b) \ast_{P} h^c\right)(d)
      \end{align*}
      and similarly for the right hand side, we get
      \begin{align*}
            \int^{x \in \AC}{P(b,c,x) \otimes_{\Z} P(a,x,d)}
            \simeq \left(h^a \ast_{P} (h^b \ast_{P} h^c)\right)(d) \, .
      \end{align*}
      In other words, this additional datum can be interpreted as the restriction of an 
      associator of $\ast_{P}$ to its values on representables,
      which is exactly the idea behind the additional data for tensor products in Freyd categories that we gave in Subsection \ref{subsection:defining_pre_structures}.
      Analogously, there are also additional data within the definition of a promonoidal category 
      which can be interpreted as unitors restricted to representables.
\end{remark}

\section{Applications and examples}

\subsection{Examples in the category of finitely presented modules}

For any ring $R$, we denote by $R\fpmodl$ the category of finitely presented left $R$-modules.
Its full subcategory generated by the free modules $R^{1 \times n}$, $n \in \N_0$
is denoted by $\Rows_R$. Note that we have $\Freyd( \Rows_R ) \simeq R\fpmodl$ (see \cite{PosFreyd}).

\begin{example}[The necessity of prostructures]
      Let $R$ be a commutative ring. Any finitely presented $R$-module $M$ 
      gives rise to a right exact bilinear functor
      \[
            T: R\fpmodl \times R\fpmodl \rightarrow R\fpmodl: (A,B) \mapsto A \otimes_R M \otimes_R B \, .
      \]
      Moreover, the associator of $\otimes_R$ defines an associator of $T$
      due to Mac Lane's coherence theorem \cite{MLCWM}.
      In other words, $R\fpmodl$ gets the structure of a semimonoidal
      category with $T$ as its tensor product.
      The restriction of the defining semimonoidal data to the subcategory $\Rows_R$
      defines a prosemimonoidal structure on $\Rows_R$ 
      (where we use the identification $\Freyd( \Rows_R ) \simeq R\fpmodl$)
      whose tensor product of two objects in $\Rows_R$
      lies outside of $\Rows_R$
      whenever $M$ is not a row module.
      
      Moreover, whenever we have a right adjoint $\IHom(A,-)$ of $(-\otimes_R A)$
      in $R\fpmodl$, \eg, when $R$ is a coherent ring, then we can calculate
      \begin{align*}
            \Hom_R( T(A,B), C ) &\simeq \Hom_R( A \otimes_R M \otimes_R B, C ) \\
            &\simeq \Hom_R\left( A , \IHom( M \otimes_R B, C ) \right)
      \end{align*}
      and see that $\IHom( M \otimes_R B, - )$ is right adjoint to $T(-,B)$.
      Restricting this right adjoint functor to $\Rows_R$ provides an example
      of a prointernal hom functor that in general does not map to $\Rows_R$.
\end{example}

\begin{example}[Internal homs cannot always be extended]
      We give an example of a closed monoidal category $\AC$
      such that the induced monoidal structure on $\Freyd( \AC )$ does not have an internal hom.
      By \Cref{theorem:homA}, it is necessary that such an $\AC$ must not have weak kernels.
      To this end, we start by looking at the ring
      \[ R = \mathbb{Q} \left[ x_i, z \, | \, i \in \mathbb{N} \right] / \langle x_i \cdot z \, | \, i \in \mathbb{N} \rangle \, . \]
      Now consider the $R$-module morphism $\alpha \colon R \xrightarrow{z} R$. 
      It is easy to see that its kernel is generated by the classes of all $x_i$, $i \in \N$
      and that this is not an \fp object. 
      Hence, the ring $R$ is not coherent. Consequently, the additive category $\AC = \Rows_R$ does not admit weak kernels either. 
      However, it is a monoidal, symmetric and closed category by the usual tensor product $\otimes_R$ of $R$-modules.
      
      We will now prove that the induced tensor product on the Freyd category $\Freyd( \AC )$,
      which can be identified with the usual tensor product of $R$-modules in $R\fpmodl$, does not admit a right adjoint. To see this, we first note
      that the $R$-module
      \begin{equation}
      \mathrm{Hom}_R \left( R / \langle z \rangle, R \right) = \left\{ 
      \begin{tikzpicture}[baseline=(current bounding box.center)]
      \def\w{2.3};
      \def\h{1.5};
      \node (A) at (0,0) {$R$};
      \node (B) at (0,-\h) {$R$};
      \node (C) at (\w,0) {$0$};
      \node (D) at (\w,-\h) {$R$};
      \draw[->,thick] (A) --node[left]{$z$} (B);
      \draw[->,thick] (B) -- node[above]{$\alpha$} (D);
      \draw[->,thick] (C) -- (D);
      \draw[->,thick,dashed] (A) -- (C);
      \end{tikzpicture} \, , \, \text{s.\,t. } z \cdot \alpha = 0 \right\} \cong \left\langle \left\{ x_i \, | \, i \in \mathbb{N} \right\} \right\rangle \label{equ:HomNotFinitelyPresented}
      \end{equation}
      is not finitely presented.
      Now assume, that there was an internal hom functor $\underline{\mathrm{Hom}}_R$ on $R\fpmodl$, so in particular
      \[ \underline{\mathrm{Hom}}_R \left( R / \langle z \rangle , R \right) = \cokernel \left( R^{1 \times a} \xleftarrow{M} R^{1 \times b} \right) \]
      for some matrix $M \in R^{b \times a}$.
      Then, by the tensor-hom-adjunction, we would have $R$-module morphisms
      \begin{align*}
            \mathrm{Hom}_R \left( R / \langle z \rangle, R \right) &\cong \mathrm{Hom}_R \left( 1, \underline{\mathrm{Hom}}_R \left( R / \langle z \rangle, R \right) \right) \\
            &\cong \mathrm{Hom}_R \left( R, \cokernel \left( R^{1 \times a} \xleftarrow{M} R^{1 \times b} \right) \right) \cong \cokernel \left( R^{1 \times a} \xleftarrow{M} R^{1 \times b} \right) \, ,
      \end{align*}
      a contradiction to \Cref{equ:HomNotFinitelyPresented}.
\end{example}

\subsection{Monoidal structures of iterated Freyd categories}

Let $\AC$ be an additive category.
An \textbf{iterated Freyd category of $\AC$} is a category contained in the following
inductively defined set:
\begin{enumerate}
      \item $\AC$ is an iterated Freyd category,
      \item if $\XC$ is an iterated Freyd category, then so are $\XC^{\op}$ and $\Freyd( \XC )$.
\end{enumerate}

\begin{theorem}
      If $\AC$ is an additive closed monoidal category with weak kernels and weak cokernels,
      then so are all iterated Freyd categories of $\AC$.
\end{theorem}
\begin{proof}[Proof by induction]
The initial case $\AC$ holds by assumption.
      Moreover, if $\XC$ is a closed monoidal category
      with weak kernels and weak cokernels, then the same applies
      to $\XC^{\op}$.
      Next, $\Freyd( \XC )$ always has cokernels, and it has kernels iff $\XC$ has weak kernels.
      Last, we can lift the monoidal closed structure from $\XC$ to $\Freyd( \XC )$ by the results presented in Subsection \ref{subsec:LiftsOfPreMonoidalStructures}.
\end{proof}

\subsection{Free abelian categories}

Let $\AC$ be an additive category.
In \cite[Theorem 6.1]{BelFredCats},
it is shown that we have an equivalence of iterated Freyd categories
\[
      \Freyd( \Freyd( \AC^{\op} )^{\op} ) \simeq \Freyd( \Freyd( \AC )^{\op} )^{\op}
\]
and that this category is actually the \textbf{free abelian category} of $\AC$.
In other words, the canonical full and faithful functors
\[
      \AC \longrightarrow \Freyd( \Freyd( \AC^{\op} )^{\op} ) \, , \hspace{3em} \AC \longrightarrow \Freyd( \Freyd( \AC )^{\op} )^{\op}
\]
are both universal functors of $\AC$ into an abelian category.
Thus, the theory of tensor products for Freyd categories also yields 
some tensor products on free abelian categories.
A nice description of free abelian categories is given by Adelman in \cite{Adelman}.

Free abelian categories are of interest in various parts of mathematics,
\eg, for dealing with so-called pp-pairs in model theory \cite{PrestPSL}.
In \cite{TensorNori}, it is shown how to lift monoidal structures on $\AC$
to free abelian categories as one step of introducing tensor products of Nori motives.
Our findings in \Cref{sec:FPMonoidalStructuresOnFreyd}
yield the following result.
\begin{theorem}
      Every promonoidal structure on $\Freyd( \AC^{\op} )^{\op}$
      extends to a monoidal structure 
      on the free abelian category of $\AC$
      whose tensor product is componentwise right exact.
\end{theorem}

As a corollary, we also get an extension of a given monoidal structure on $\AC$ to the free abelian category
an thus recover the induced functor defined in \cite[Proposition 1.3]{TensorNori}
by applying the following steps:
\begin{enumerate}
      \item Take the induced monoidal structure on $\AC^{\op}$. Interpret it as an \fp promonoidal structure.
      \item This induces a right exact monoidal structure on $\Freyd( \AC^{\op} )$. 
      \item Interpret the induced monoidal structure on $\Freyd( \AC^{\op} )^{\op}$ as an \fp promonoidal structure.
      \item This induces a right exact monoidal structure on $\Freyd( \Freyd( \AC^{\op} )^{\op} )$, hence on the free abelian category of $\AC$.
\end{enumerate}

\subsection{Implementations of monoidal structures} \label{sec:MonoidalConstructions}
Our results yield a unified approach to the implementation of monoidal structures for \fp modules, \fp graded modules, \fp functors, that has in parts been realized in \cite{FreydCategoriesPackage}. 
Let us therefore explicitly present the underlying constructions. 

\begin{notationnonumber}
      For a protensor product $\T \colon \AC \times \AC \to \Freyd( \AC )$ and objects $a,b \in \AC$,
      we write
      \[
            \T( a,b ) = \left( g_T( a,b) \xleftarrow{\rho_T( a,b )} r_T( a,b ) \right) \in \Freyd( \AC ) \, .
      \]
      For morphisms $a \xrightarrow{\alpha} a^\prime$ and $b \xrightarrow{\beta} b^\prime$ in $\AC$, we represent the morphism $\T( \alpha, \beta ) \in \Mor ( \Freyd( \AC ) )$ by the commutative diagram
      \begin{center}
            \begin{tikzpicture}[baseline=(current bounding box.center),label/.style={postaction={
                      decorate,
                      decoration={markings, mark=at position .5 with \node #1;}},
                      mylabel/.style={thick, draw=black, align=center, minimum width=0.5cm, minimum height=0.5cm,fill=white}}]
            
                      \def\w{4};
                      \def\h{3};
            
                      \node (C) at (\w,\h) {$r_T( a,b)$};
                      \node (A) at (-\w,\h) {$g_T( a,b)$};
                      \node (B) at (-\w,0) {$g_T( a^\prime, b^\prime )$};
                      \node (D) at (\w,0) {$r_T( a^\prime, b^\prime )$};
                      
                      \draw[->,thick,dashed] (C) --node[left]{$\omega_T( \alpha, \beta )$} (D);
                      \draw[->,thick] (A) --node[right]{$\delta_T( \alpha, \beta)$} (B);
                      \draw[->,thick] (C) --node[below]{$\rho_T( a, b )$} (A);
                      \draw[->,thick] (D) --node[above]{$\rho_T( a^\prime, b^\prime )$} (B);
            
                      \node at (0.0*\w,0.5*\h) {$\circlearrowleft$};
            
            \end{tikzpicture}
            \end{center}
            Finally, let us agree that we denote a given protensor unit of $\AC$ by $1 = ( g_1 \xleftarrow{\rho_1} r_1 )$.
\end{notationnonumber}

\subsubsection{Tensor product of morphisms}

\begin{construction} \label{construc:TensorProductOfMorphisms}
The tensor product of two morphism $A \xrightarrow{\{\alpha, \omega_{\alpha}\}} A'$ and $B \xrightarrow{\{\beta, \omega_{\beta}\}} B'$ in $\Freyd( \AC )$ is represented by the following commutative diagram:
\begin{center}
\begin{tikzpicture}[baseline=(current bounding box.center),label/.style={postaction={
          decorate,
          decoration={markings, mark=at position .5 with \node #1;}},
          mylabel/.style={thick, draw=black, align=center, minimum width=0.5cm, minimum height=0.5cm,fill=white}}]

          \def\w{3};
          \def\h{3};

          \node (A2) at (-2*\w,\h) {$A \oF B$};
          \node (B2) at (-2*\w,0) {$A^\prime \oF B^\prime$};
          \draw[->,thick] (A2) --node[left]{$\{\alpha\} \oF \{\beta\}$} (B2);

          \draw[<->, thick, dotted] (-1.7*\w,\h) -- (-1.3*\w,\h);
          \draw[<->, thick, dotted] (-1.7*\w,0) -- (-1.3*\w,0);

          \node (A) at (-\w,\h) {$g_T( a,b )$};
          \node (B) at (-\w,0) {$g_T( a^\prime, b^\prime )$};
          \node (C) at (\w,\h) {$r_T( a,b ) \oplus g_T ( a, r_b ) \oplus g_T ( r_a, b )$};
          \node (D) at (\w,0) {$r_T( a^\prime, b^\prime ) \oplus g_T ( a^\prime, r_{b^\prime} ) \oplus g_T( r_{a^\prime}, b^\prime )$};
          \draw[->,thick,dashed] (C) --node[right] {$\left( \begin{smallmatrix} \omega_T( \alpha, \beta ) & \cdot & \cdot \\ \cdot & \delta_T( \alpha, \omega_\beta ) & \cdot \\ \cdot & \cdot & \delta_T( \omega_\alpha, \beta ) \end{smallmatrix} \right)$} (D);
          \draw[->,thick] (C) --node[below]{$\left( \begin{smallmatrix} \rho_T( a,b ) \\ \delta_T( \mathrm{id}_a, \rho_b ) \\ \delta_T( \rho_a, \mathrm{id}_b ) \end{smallmatrix} \right)$} (A);
          \draw[->,thick] (D) --node[above]{$\left( \begin{smallmatrix} \rho_T( a^\prime,b^\prime ) \\ \delta_T( \mathrm{id}_{a^\prime}, \rho_{b^\prime} ) \\ \delta_T( \rho_{a^\prime}, \mathrm{id}_{b^\prime} ) \end{smallmatrix} \right)$} (B);
          \draw[->,thick] (A) --node[left]{$\delta_T( \alpha, \beta )$} (B);

          \node at (-0.2*\w,0.5*\h) {$\circlearrowleft$};

\end{tikzpicture}
\end{center}
\end{construction}

\begin{proof}[Correctness of construction]
We apply \Cref{construction:extension_of_multi_functor} to the protensor product $\T$. 
Thereby, we observe that $\{\alpha, \omega_{\alpha}\} \oF \{\beta, \omega_{\beta}\}$ can be constructed as 
the induced morphism between the cokernels of the horizontal arrows in the following diagram:
\begin{center}
\begin{tikzpicture}[baseline=(current bounding box.center),label/.style={postaction={
          decorate,
          decoration={markings, mark=at position .5 with \node #1;}},
          mylabel/.style={thick, draw=black, align=center, minimum width=0.5cm, minimum height=0.5cm,fill=white}}]

          \def\w{4};
          \def\h{2.5};

          \node (A) at (-\w,\h) {$\T(a,b)$};
          \node (B) at (-\w,0) {$\T(a^\prime, b^\prime)$};
          \node (C) at (\w,\h) {$\T(a,r_b) \oplus \T(r_a, b)$};
          \node (D) at (\w,0) {$\T(a^\prime, r_{b^\prime} ) \oplus \T(r_{a^\prime}, b^\prime )$};

          \draw[->,thick,dashed] (C) --node[right] {$\left( \begin{smallmatrix} \T( \delta_\alpha, \omega_{\beta} ) & \cdot \\ \cdot & \T( \omega_\alpha, \delta_\beta ) \end{smallmatrix} \right)$} (D);
          \draw[->,thick] (C) --node[below]{$\left( \begin{smallmatrix} \T( \mathrm{id}_a, \rho_b ) \\ \T( \rho_a, \mathrm{id}_b )\end{smallmatrix} \right)$} (A);
          \draw[->,thick] (D) --node[above]{$\left( \begin{smallmatrix} \T( \mathrm{id}_{a^\prime}, \rho_{b^\prime} ) \\ \T( \rho_{a^\prime}, \mathrm{id}_{b^\prime} ) \end{smallmatrix} \right)$} (B);
          \draw[->,thick] (A) --node[right]{$\T( \delta_\alpha, \delta_\beta )$} (B);

          \node at (0.0*\w,0.5*\h) {$\circlearrowleft$};

\end{tikzpicture}
\end{center}
We can express this diagram purely in terms of objects and morphisms of $\AC$, namely
\begin{center}
\begin{tikzpicture}[baseline=(current bounding box.center),label/.style={postaction={
          decorate,
          decoration={markings, mark=at position .5 with \node #1;}},
          mylabel/.style={thick, draw=black, align=center, minimum width=0.5cm, minimum height=0.5cm,fill=white}}]

          \def\w{4.5};
          \def\h{3.5};

          \node (A) at (-\w,\h) {$\left( g_T( a,b ) \xleftarrow{\rho_T( a,b )} r_T( a,b ) \right)$};
          \node (B) at (-\w,0) {$\left( g_T( a^\prime,b^\prime ) \xleftarrow{\rho_T( a^\prime,b^\prime )} r_T( a^\prime,b^\prime ) \right)$};
          \node (C) at (\w,\h) {$\left( \begin{smallmatrix} g_T( a, r_b ) \\ \oplus \\ g_T( r_a, b ) \end{smallmatrix} \xleftarrow{\left( \begin{smallmatrix} \rho_T( a, r_b ) & \cdot \\ \cdot & \rho_T( r_a, b ) \end{smallmatrix}\right)} \begin{smallmatrix} r_T( a, r_b ) \\ \oplus \\ r_T( r_a, b ) \end{smallmatrix} \right)$};
          \node (D) at (\w,0) {$\left( \begin{smallmatrix} g_T( a^\prime, r_{b^\prime} ) \\ \oplus \\ g_T( r_{a^\prime}, b^\prime ) \end{smallmatrix} \xleftarrow{\left( \begin{smallmatrix} \rho_T( a^\prime, r_{b^\prime} ) & \cdot \\ \cdot & \rho_T( r_{a^\prime}, b^\prime ) \end{smallmatrix}\right)} \begin{smallmatrix} r_T( a^\prime, r_{b^\prime} ) \\ \oplus \\ r_T( r_{a^\prime}, b^\prime ) \end{smallmatrix} \right)$};

          \draw[->,thick,dashed] (C) --node[right] {$\left\{ \left( \begin{smallmatrix} \delta_T ( \delta_\alpha, \omega_{\beta} ) & \cdot \\ \cdot & \delta_T ( \omega_\alpha, \delta_\beta ) \end{smallmatrix} \right) \right\}$} (D);
          \draw[->,thick] (C) --node[below]{$\left\{ \left( \begin{smallmatrix} \delta_T( \mathrm{id}_a, \rho_b ) \\ \delta_T( \rho_a, \mathrm{id}_b ) \end{smallmatrix} \right) \right\}$} (A);
          \draw[->,thick] (D) --node[above]{$\left\{ \left( \begin{smallmatrix} \delta_T( \mathrm{id}_{a^\prime}, \rho_{b^\prime} ) \\ \delta_T( \rho_{a^\prime}, \mathrm{id}_{b^\prime} ) \end{smallmatrix} \right) \right\}$} (B);
          \draw[->,thick] (A) --node[right]{$\left\{ \delta_T( \alpha, \beta ) \right\}$} (B);

          \node at (0.0*\w,0.5*\h) {$\circlearrowleft$};
\end{tikzpicture}
\end{center}
Now, the claim follows by the explicit constructions of cokernels in Freyd categories given in \Cref{construction:cokernels_in_freyd}.
\end{proof}

\subsubsection{Tensor unit and unitors}

\begin{construction}[Tensor unit]
The tensor unit $\unitF$ of $\Freyd( \AC )$ is identical to the protensor unit $1 = ( g_1 \xleftarrow{\rho_1} r_1 )$ of $\AC$.
\end{construction}
\begin{proof}[Correctness of construction]
See Subsubsection \ref{subsubsection:unitors}.
\end{proof}

Recall that we denote the \fp left prounitor by $\LU_a$ for $a \in \AC$.

\begin{construction}[Left unitor]
For $A \in \Freyd( \AC )$, the left unitor 
\[ \LUhat_A \colon \unitF \oF A \xrightarrow{\sim} A \]
is given by the following commutative diagram
\begin{center}
\begin{tikzpicture}[baseline=(current bounding box.center),label/.style={postaction={
          decorate,
          decoration={markings, mark=at position .5 with \node #1;}},
          mylabel/.style={thick, draw=black, align=center, minimum width=0.5cm, minimum height=0.5cm,fill=white}}]

          \def\w{6.5};
          \def\h{2.5};

          \node (A2) at (-2*\w,0) {$\unitF \oF A$};
          \node (B2) at (-2*\w,-\h) {$A$};
          \draw[->,thick] (A2) --node[left]{$\LUhat_A$} (B2);
          
          \draw[<->, thick, dotted] (-1.75*\w,0) -- (-1.25*\w,0);
          \draw[<->, thick, dotted] (-1.75*\w,-\h) -- (-1.25*\w,-\h);
          
          \node (A) at (-\w,0) {$g_T( g_1, a )$};
          \node (B) [align=right] at (0,0) {$r_T( g_1, a )$ \\ $\oplus g_T( r_1, a )$ \\ $\oplus g_T( g_1,r_a )$};
          \node (C) at (-\w,-\h) {$a$};
          \node (D) at (0,-\h) {$r_a$};
          \draw[->,thick] (B) --node[above]{$\left( \begin{smallmatrix} \rho_T( g_1, a ) \\ \delta_T( \rho_1, \mathrm{id}_a ) \\ \delta_T( \mathrm{id}_{g_1}, \rho_a ) \end{smallmatrix} \right)$} (A);
          \draw[->,thick] (D) --node[below]{$\rho_a$} (C);
          \draw[->,thick] (A) --node[left]{$\delta_{\LU_a}$} (C);
          \draw[->,thick,dashed] (B) --node[left]{$\left( \begin{smallmatrix} 0 \\ 0 \\ \delta_{\LU_{r_a}} \end{smallmatrix} \right)$} (D);

          \node at (-0.50*\w,-0.5*\h) {$\circlearrowleft$};

\end{tikzpicture}
\end{center}
\end{construction}

\begin{proof}[Correctness of construction]
The left unitor $\LUhat_A$ is understood as the unique morphism which turns the following diagram into a commutative diagram with exact columns.
\begin{center}
\begin{tikzpicture}[label/.style={postaction={
          decorate,
          decoration={markings, mark=at position .5 with \node #1;}},
          mylabel/.style={thick, draw=black, align=center, minimum width=0.5cm, minimum height=0.5cm,fill=white}}]

          \def\w{6.5};
          \def\h{2.5};

          \node (A) at (-\w,0) {$\left( g_T( g_1, r_a ) \xleftarrow{\left( \begin{smallmatrix} \rho_T( g_1, r_a ) \\ \delta_T( \rho_1, \mathrm{id}_{r_a} ) \end{smallmatrix} \right)} r_T( g_1, r_a ) \oplus g_T( r_1, r_a ) \right)$};
          \node (B) at (0,0) {$( r_a \xleftarrow{} 0 )$};
          \node (C) at (-\w,-\h) {$\left( g_T( g_1, a ) \xleftarrow{\left( \begin{smallmatrix} \rho_T( g_1, a ) \\ \delta_T( \rho_1, \mathrm{id}_a ) \end{smallmatrix} \right)} r_T( g_1, a ) \oplus g_T( r_1, a ) \right)$};
          \node (D) at (0,-\h) {$( a \xleftarrow{} 0 )$};
          \node (E) at (-\w,-2*\h) {$1 \oF A$};
          \node (F) at (0,-2*\h) {$A$};
          \node (G) at (-\w,-2.5*\h) {$0$};
          \node (H) at (0,-2.5*\h) {$0$};

          \draw[->,thick] (A) --node[below]{$\LU_{r_a}$} (B);
          \draw[->,thick] (C) --node[above]{$\LU_{a}$} (D);
          \draw[->,thick] (A) --node[left]{$\{ \delta_T( \mathrm{id}_{g_1}, \rho_a ) \}$} (C);
          \draw[->,thick] (B) --node[right]{$\{ \rho_a \}$} (D);
          \draw[->,thick] (E) --node[above]{$\LUhat_A$} (F);
          \draw[->,thick] (D) -- (F);
          \draw[->,thick] (E) -- (G);
          \draw[->,thick] (C) -- (E);
          \draw[->,thick] (F) -- (H);

          \node at (-0.50*\w,-0.5*\h) {$\circlearrowleft$};
          \node at (-0.50*\w,-1.5*\h) {$\circlearrowleft$};

\end{tikzpicture} \qedhere
\end{center}
\end{proof}

\begin{construction}[Left unitor inverse]
For $A \in \Freyd( \AC )$, the left unitor inverse 
\[ \LUhat_A^{-1} \colon  A \xrightarrow{\sim} \unitF \oF A \]
is given by the following commutative diagram
\begin{center}
\begin{tikzpicture}[baseline=(current bounding box.center),label/.style={postaction={
          decorate,
          decoration={markings, mark=at position .5 with \node #1;}},
          mylabel/.style={thick, draw=black, align=center, minimum width=0.5cm, minimum height=0.5cm,fill=white}}]

          \def\w{6.5};
          \def\h{2.5};

          \node (A2) at (-2*\w,0) {$A \oF \unitF$};
          \node (B2) at (-2*\w,-\h) {$A$};
          \draw[->,thick] (B2) --node[left]{$\LUhat_A^{-1}$} (A2);
          
          \draw[<->, thick, dotted] (-1.75*\w,0) -- (-1.25*\w,0);
          \draw[<->, thick, dotted] (-1.75*\w,-\h) -- (-1.25*\w,-\h);
          
          \node (A) at (-\w,0) {$g_T( g_1, a )$};
          \node (B) [align=right] at (0,0) {$r_T( g_1, a )$ \\ $\oplus g_T( r_1, a )$ \\ $\oplus g_T( g_1,r_a )$};
          \node (C) at (-\w,-\h) {$a$};
          \node (D) at (0,-\h) {$r_a$};
          \draw[->,thick] (B) --node[above]{$\left( \begin{smallmatrix} \rho_T( g_1, a ) \\ \delta_T( \rho_1, \mathrm{id}_a ) \\ \delta_T( \mathrm{id}_{g_1}, \rho_a ) \end{smallmatrix} \right)$} (A);
          \draw[->,thick] (D) --node[below]{$\rho_a$} (C);
          \draw[->,thick] (C) --node[left]{$\delta_{\LU_a^{-1}}$} (A);
          \draw[->,thick,dashed] (D) --node[left]{$\left( \begin{smallmatrix} 0 & 0 & \delta_{\LU_{r_a}^{-1}} \end{smallmatrix} \right)$} (B);

          \node at (-0.50*\w,-0.5*\h) {$\circlearrowleft$};

\end{tikzpicture}
\end{center}
\end{construction}

\begin{proof}[Correctness of construction]
The same reasoning as for the left unitor applies.
\end{proof}

\begin{remark}
The constructions of the right unitor and its inverse follow analogously.
\end{remark}

\subsubsection{Associator}

Recall that we denote the \fp proassociator by $\Ass_{a,b,c}$ for $a,b,c \in \AC$.

\begin{construction}[Associator]
Given $A,B,C \in \Freyd( \AC )$, the associator 
\[ \Asshat_{A,B,C} \colon  A \oF \left( B \oF C \right) \xrightarrow{\sim} \left( A \oF B \right) \oF C \]
is given by the following commutative diagram, in which $\omega_{\Ass_{a,b,c}}$ denotes the $2 \times 2$ witness-matrix of the \fp proassociator $\Ass_{a,b,c}$:
\begin{center}
\begin{tikzpicture}[baseline=(current bounding box.center),label/.style={postaction={
          decorate,
          decoration={markings, mark=at position .5 with \node #1;}},
          mylabel/.style={thick, draw=black, align=center, minimum width=0.5cm, minimum height=0.5cm,fill=white}}]

          \def\w{3.5};
          \def\h{3.0};

          \node (A2) at (-1.3*\w,-3*\h) {$\left( A \oF B \right) \oF C$};
          \node (B2) at (-1.3*\w,-\h) {$A \oF \left( B \oF C \right)$};
          \draw[->,thick] (B2) --node[left]{$\Asshat_{A,B,C}$} (A2);
          
          \draw[<->, thick, dotted] (-0.88*\w,-\h) -- (-0.4*\w,-\h);
          \draw[<->, thick, dotted] (-0.88*\w,-3*\h) -- (-0.4*\w,-3*\h);
          
          \node (A) at (0,-3*\h) {$g_T( g_T( a,b ), c )$};
          \node (B) at (0,-\h) {$g_T( a, g_T( b,c ) )$};
          \node (C) [align=center] at (2*\w,-3*\h) {$r_T( g_T( a,b ), c )$ \\ $\oplus g_T( r_T( a,b ), c )$ \\ $\oplus g_T ( g_T( a,b ), r_c )$ \\ $\oplus g_T( g_T( a, r_b ), c )$ \\ $\oplus g_T( g_T( r_a, b ), c )$};
          \node (D) [align=center] at (2*\w,-\h) {$r_T( a, g_T( b,c ) )$ \\ $\oplus g_T ( a, r_T( b, c ) )$ \\ $\oplus g_T( a, g_T( b, r_c ) )$ \\ $\oplus g_T( a, g_T( r_b, c ) )$ \\ $\oplus g_T( r_a, g_T( b, c ) )$};

          \draw[->,thick] (B) --node[left]{$\delta_{\Ass_{a,b,c}}$} (A);
          \draw[->,thick] (C) --node[above]{$\left( \begin{smallmatrix} \rho_T( g_T( a,b ), c ) \\ \delta_T( \rho_T( a,b ), \mathrm{id}_c ) \\ \delta_T( g_T( a,b ), \rho_c ) \\ \delta_T( \delta_T ( \mathrm{id}_a, \rho_b ), \mathrm{id}_c ) \\ \delta_T( \delta_T ( \rho_a, \mathrm{id}_b ), \mathrm{id}_c ) \end{smallmatrix} \right)$} (A);
          \draw[->,thick] (D) --node[below]{$\left( \begin{smallmatrix} \rho_T( a, g_T( b,c ) ) \\ \delta_T( \mathrm{id}_a, \rho_T( b, c ) ) \\ \delta_T( \mathrm{id}_a, \delta_T( \mathrm{id}_b, \rho_c ) ) \\ \delta_T( \mathrm{id}_a, \delta_T ( \rho_b, \mathrm{id}_c ) ) \\ \delta_T( \rho_a, \mathrm{id}_{g_T( b,c )} ) \end{smallmatrix} \right)$} (B);
          \draw[->,thick,dashed] (D) --node{$\left( \begin{smallmatrix} \omega_{\Ass_{a,b,c}} \\ & & \delta_{\Ass_{a,b,r_c}} \\ & & & \delta_{\Ass_{a,r_b, c}} \\ & & & & \delta_{\Ass_{r_a, b, c}} \end{smallmatrix} \right)$} (C);
          
          \node at (0.95*\w,-2.0*\h) {$\circlearrowleft$};

\end{tikzpicture}
\end{center}
\end{construction}

\begin{proof}[Correctness of construction]
Let us first define the functors
\begin{align*}
A_R & \colon \AC \times \AC \times \AC \to \Freyd( \AC ) \, , \, (a,b,c) \mapsto a \oF ( b \oF c ) \, , \\
A_L & \colon \AC \times \AC \times \AC \to \Freyd( \AC ) \, , \, (a,b,c) \mapsto ( a \oF b ) \oF c
\end{align*}
and recall that the proassociator is a natural transformation $\Ass \colon A_R \Rightarrow A_L$. Hence, by applying \Cref{construction:extension_of_multi_nat} we extend the proassociator to a natural transformation $\Asshat \colon \widehat{A}_R \Rightarrow \widehat{A}_L$. Its component at $A,B,C \in \Freyd( \AC )$ can be constructed as the induced morphism
between the cokernels of the horizontal arrows in the following diagram:
\begin{center}
\begin{tikzpicture}[baseline=(current bounding box.center),label/.style={postaction={
          decorate,
          decoration={markings, mark=at position .5 with \node #1;}},
          mylabel/.style={thick, draw=black, align=center, minimum width=0.5cm, minimum height=0.5cm,fill=white}}]

          \def\w{7};
          \def\h{3};

          \node (B2) at (0,0) {$A_R( a,b,c )$};
          \node (B3) at (\w,0) {$A_R( a,b,r_c ) \oplus A_R( a, r_b, c ) \oplus A_R( r_a, b, c )$};

          \node (G2) at (0,-\h) {$A_L( a,b,c )$};
          \node (G3) at (\w,-\h) {$A_L( a,b,r_c ) \oplus A_L( a,r_b,c ) \oplus A_L( r_a,b,c )$};

          \draw[->,thick] (B3) --node[below]{$\left( \begin{smallmatrix} A_R( a,b,\rho_c) \\ A_R( a,\rho_b,c) \\ A_R( \rho_a, b, c)\end{smallmatrix} \right)$} (B2);
          \draw[->,thick] (G3) --node[above]{$\left( \begin{smallmatrix} A_L( a,b,\rho_c) \\ A_L( a,\rho_b,c) \\ A_L( \rho_a, b, c)\end{smallmatrix} \right)$} (G2);
          \draw[->,thick] (B2) --node[left]{$\Ass_{a,b,c}$} (G2);
          \draw[->,thick] (B3) --node[right]{$\left( \begin{smallmatrix} \Ass_{a,b,r_c} \\ & \Ass_{a,r_b,c} \\ & & \Ass_{r_a,b,c}\end{smallmatrix} \right)$} (G3);

          \node at (0.5*\w,-0.5*\h) {$\circlearrowleft$};

\end{tikzpicture}
\end{center}
Again, the claim follows by the explicit constructions of cokernels in Freyd categories given
in \Cref{construction:cokernels_in_freyd}.
\end{proof}

\subsubsection{Braiding}

Recall that we denote the \fp probraiding by $\Br_{a,b}$ for $a,b, \in \AC$.

\begin{construction}
For objects $A,B \in \Freyd( \AC )$, the braiding morphism 
\[ \Brhat_{A,B} \colon A \oF B \to B \oF A \]
is given by the following diagram
\begin{center}
\begin{tikzpicture}[baseline=(current bounding box.center),label/.style={postaction={
          decorate,
          decoration={markings, mark=at position .5 with \node #1;}},
          mylabel/.style={thick, draw=black, align=center, minimum width=0.5cm, minimum height=0.5cm,fill=white}}]

          \def\w{3.5};
          \def\h{3};

          \node (A) at (-\w,\h) {$g_T( a,b )$};
          \node (B) at (-\w,0) {$g_T( b, a )$};
          \node (C) at (\w,\h) {$r_T( a,b ) \oplus g_T ( a, r_b ) \oplus g_T ( r_a, b )$};
          \node (D) at (\w,0) {$r_T( b, a ) \oplus g_T ( b, r_a ) \oplus g_T( r_b, a )$};
          \node (E) at (-2*\w,0) {$B \oF A$};
          \node (F) at (-2*\w,\h) {$A \oF B$};
          
          \draw[->,thick,dashed] (C) --node[right] {$\left( \begin{smallmatrix} \omega_{\Br_{a,b}} & 0 & 0 \\ 0 & 0 & \delta_{\Br_{a,r_b}} \\ 0 & \delta_{\Br_{r_a,b}} & 0  \end{smallmatrix} \right)$} (D);
          \draw[->,thick] (C) --node[below]{$\left( \begin{smallmatrix} \rho_T( a,b ) \\ \delta_T( \mathrm{id}_a, \rho_b ) \\ \delta_T( \rho_a, \mathrm{id}_b ) \end{smallmatrix} \right)$} (A);
          \draw[->,thick] (D) --node[above]{$\left( \begin{smallmatrix} \rho_T( b,a ) \\ \delta_T( \mathrm{id}_{b}, \rho_{a} ) \\ \delta_T( \rho_{b}, \mathrm{id}_{a} ) \end{smallmatrix} \right)$} (B);
          \draw[->,thick] (A) --node[left]{$\delta_{\Br_{a,b}}$} (B);
          \draw[->,thick] (F) --node[left]{$\widehat{\Br}_{A,B}$} (E);
          \draw[<->,thick,dotted] (-1.65*\w,\h) -- (-1.25*\w,\h);
          \draw[<->,thick,dotted] (-1.65*\w,0) -- (-1.25*\w,0);
          
          \node at (0.6*\w,0.5*\h) {$\circlearrowleft$};

\end{tikzpicture}
\end{center}
\end{construction}

\begin{proof}[Correctness of construction]
Let us define the two functors
\[ \T_L \colon \AC \times \AC \to \AC \, , \, (a,b) \mapsto \T( a,b ) \, , \qquad \T_R \colon \AC \times \AC \to \AC \, , \, (a,b) \mapsto \T( b,a ) \, . \]
To extend the braiding $\Br \colon \T_L \Rightarrow \T_R$, we apply \Cref{construction:extension_of_multi_nat}. 
This shows that the component $\Brhat_{A,B}$ can be constructed as the induced morphism
between the cokernels of the horizontal arrows in the following diagram:
\begin{center}
\begin{tikzpicture}[baseline=(current bounding box.center),label/.style={postaction={
          decorate,
          decoration={markings, mark=at position .5 with \node #1;}},
          mylabel/.style={thick, draw=black, align=center, minimum width=0.5cm, minimum height=0.5cm,fill=white}}]

          \def\w{9};
          \def\h{2.5};

          \node (A) at (0,0) {$T_L( a, b )$};
          \node (B) at (\w,0) {$T_L(a, r_b ) \oplus T_L( r_a, b )$};
          \node (C) at (0,-\h) {$T_R( a, b )$};
          \node (D) at (\w,-\h) {$T_R( a, r_b ) \oplus T_R(r_a, b )$};

          \draw[->,thick] (B) --node[above]{$\pmatrow{\T_L \left( \mathrm{id}_a, \rho_b \right)}{\T_L \left( \rho_a, \mathrm{id}_b \right)}$} (A);
          \draw[->,thick] (D) --node[above]{$\pmatrow{\T_R \left( \mathrm{id}_a, \rho_b \right)}{\T_R \left( \rho_a, \mathrm{id}_b \right)}$} (C);
          \draw[->,thick] (A) --node[right]{$\Br_{a,b}$} (C);
          \draw[->,thick] (B) --node[right]{$\begin{pmatrix}{\Br_{a,r_b}} & \cdot \\ \cdot &\Br_{r_a,b} \end{pmatrix}$} (D);

          \node at (0.5*\w,-0.5*\h) {$\circlearrowleft$};

\end{tikzpicture}
\end{center}
We can represent this diagram with objects and morphisms of the category $\AC$ only, namely:
\begin{center}
\resizebox{\textwidth}{!}{%
\begin{tikzpicture}[baseline=(current bounding box.center),label/.style={postaction={
          decorate,
          decoration={markings, mark=at position .5 with \node #1;}},
          mylabel/.style={thick, draw=black, align=center, minimum width=0.5cm, minimum height=0.5cm,fill=white}}]

          \def\w{11};
          \def\h{4};

          \node (A) at (0,0) {$\left( g_T( a, b) \xleftarrow{\rho_T( a, b)} r_T( a, b ) \right)$};
          \node (B) at (\w,0) {$\left( \begin{array}{c} g_T( a, r_b ) \\ \oplus \\ g_T( r_a, b ) \end{array} \xleftarrow{\left( \begin{smallmatrix} \rho_T( a, r_b ) & \cdot \\ \cdot & \rho_T( r_a, b ) \end{smallmatrix} \right)} \begin{array}{c} r_T( a, r_b ) \\ \oplus \\ r_T( r_a, b ) \end{array} \right)$};
          \node (C) at (0,-\h) {$\left( g_T( b, a ) \xleftarrow{\rho_T( b, a )} r_T( b, a ) \right)$};
          \node (D) at (\w,-\h) {$\left( \begin{array}{c} g_T( b, r_a ) \\ \oplus \\ g_T( r_b, a ) \end{array} \xleftarrow{\left( \begin{smallmatrix} \rho_T( b, r_a ) & \cdot \\ \cdot & \rho_T( r_b, a ) \end{smallmatrix} \right)} \begin{array}{c} r_T( b, r_a ) \\ \oplus \\ r_T( r_b, a ) \end{array} \right)$};

          \draw[->,thick] (B) --node[above]{$\left\{ \left( \begin{smallmatrix} \delta_T \left( \mathrm{id}_a, \rho_b \right) & \cdot \\ \cdot & \delta_T \left( \rho_a, \mathrm{id}_b \right) \end{smallmatrix} \right) \right\}$} (A);
          \draw[->,thick] (D) --node[above]{$\left\{ \left( \begin{smallmatrix} \delta_T \left( \mathrm{id}_b, \rho_a \right) & \cdot \\ \cdot & \delta_T \left( \rho_b, \mathrm{id}_a \right) \end{smallmatrix} \right) \right\}$} (C);
          \draw[->,thick] (A) --node[right]{$\left\{ \delta_{\Br} (a,b) \right\}$} (C);
          \draw[->,thick] (B) --node[left]{$\left\{ \left( \begin{smallmatrix} \cdot & \delta_{\Br} ( r_a,b ) \\ \delta_{\Br} ( a,r_b ) & \cdot \end{smallmatrix} \right)\right\}$} (D);

          \node at (0.5*\w,-0.5*\h) {$\circlearrowleft$};

\end{tikzpicture}}
\end{center}
This leads to the presented construction diagram.
\end{proof}

\def\cprime{$'$} \def\cprime{$'$} \def\cprime{$'$} \def\cprime{$'$}
  \def\cprime{$'$}
\providecommand{\bysame}{\leavevmode\hbox to3em{\hrulefill}\thinspace}
\providecommand{\MR}{\relax\ifhmode\unskip\space\fi MR }
\providecommand{\MRhref}[2]{%
  \href{http://www.ams.org/mathscinet-getitem?mr=#1}{#2}
}
\providecommand{\href}[2]{#2}

\end{document}